\documentclass[11pt]{article}
\usepackage{amsmath}
\usepackage{amsfonts}
\usepackage{color}
\usepackage{latexsym}
\usepackage{tablefootnote}
\usepackage{epsfig}
\usepackage{multirow}
\usepackage{float}
\usepackage{array}
\usepackage{amssymb,amsthm}
\newcommand*{\QEDA}{\hfill\hbox{\vrule width1.0ex height1.0ex}}
\usepackage{algorithm,algorithmic}

\usepackage{makecell,booktabs}
\usepackage[version=4]{mhchem}
\usepackage[strict]{changepage}

\usepackage{amsmath}
\usepackage{amsfonts}
\usepackage{latexsym}
\usepackage{amssymb}

\newtheorem{thm}{Theorem}[section]
\newtheorem{theorem}[thm]{Theorem}

\newtheorem{lemma}[thm]{Lemma}

\newtheorem{remark}[thm]{Remark}

\newcommand{\beq}{\begin{equation}}
\newcommand{\eeq}{\end{equation}}
\newcommand{\beqa}{\begin{eqnarray}}
\newcommand{\eeqa}{\end{eqnarray}}
\newcommand{\beqas}{\begin{eqnarray*}}
\newcommand{\eeqas}{\end{eqnarray*}}
\newcommand{\bi}{\begin{itemize}}
\newcommand{\ei}{\end{itemize}}

\newcommand{\vgap}{\vspace{.1in}}
\newcommand{\nn}{\nonumber}

\newcommand{\R}{\mathbb{R}}

\newcommand{\lam}{{\lambda}}

\newcommand{\inner}[2]{\langle #1,#2\rangle}

\newcommand{\argmin}{\mathrm{argmin}\,}

\newcommand{\dom}{\mathrm{dom}\,}

\newcommand{\bConv}[1]{\overline{\mbox{\rm Conv}}\,(\R^{#1})}

\newcommand{\tx}{\tilde x}

\begin{document}
\title{An Average Curvature Accelerated Composite Gradient Method \\ for Nonconvex Smooth Composite Optimization Problems}
 \date{September 9, 2019 \\ 
 1st revision: October 18, 2019 \\
 2nd revision: May 19, 2020 \\ 
 3rd revision: September 16, 2020 \\
 4th revision: October 26, 2020}
 \makeatletter
	\maketitle
	
	
\begin{center}
			\textsc{Jiaming Liang \footnote{School of Industrial and Systems
			Engineering, Georgia Institute of
			Technology, Atlanta, GA, 30332-0205.
			(email: {\tt jiaming.liang@gatech.edu}).
			This author
			was partially supported by
			ONR Grant
			N00014-18-1-2077.
			This author was also partially supported by
			NSF grant CCF-1740776 through a joint fellowship from the Algorithms \& Randomness Center (ARC) and
			the Transdisciplinary Research Institute for Advancing Data Science (TRIAD), of which the latter one belongs to the TRIPODS program at NSF and locates at Georgia Tech (http://triad.gatech.edu).}
            and 
			Renato D.C. Monteiro} \footnote{School of Industrial and Systems
			Engineering, Georgia Institute of
			Technology, Atlanta, GA, 30332-0205.
			(email: {\tt monteiro@isye.gatech.edu}). This author
			was partially supported by ONR Grant
			N00014-18-1-2077.}
		\end{center}	

\begin{abstract}
This  paper presents an accelerated composite gradient (ACG) variant, referred to as the AC-ACG method,
for solving nonconvex smooth composite minimization problems.
As opposed to well-known ACG variants that are either based on a known Lipschitz gradient constant or a sequence of maximum observed curvatures, the current one is based on the average of all past observed curvatures.
More specifically, AC-ACG uses a positive multiple of  the average of all observed curvatures until
the previous iteration as a way to estimate the ``function curvature'' at the current point and then two
resolvent evaluations to compute the next iterate. In contrast to other variable Lipschitz estimation
variants, e.g., the ones based on the maximum curvature, AC-ACG always accepts the aforementioned iterate
regardless how poor the Lipschitz estimation  turns out to be. Finally, computational results are presented to illustrate
the efficiency of AC-ACG on both randomly generated and real-world problem instances.

\vspace{.1in}


\vspace{.1in}


\end{abstract}

\section{Introduction}\label{sec:intro}
In this paper, we study an ACG-type algorithm for solving a nonconvex smooth composite optimization (SCO) problem
\begin{equation}\label{eq:PenProb2Intro}
\phi_*:=\min \left\{ \phi(z):=f(z) + h(z)  : z \in \mathbb{R}^n \right \}
\end{equation}
where $f$ is a real-valued differentiable (possibly nonconvex) function with an $M$-Lipschitz continuous gradient
on $\dom h$ and $h:\R^n \to (-\infty,\infty]$ is a proper lower semicontinuous convex function with a bounded domain.

A large class of algorithms for solving \eqref{eq:PenProb2Intro} sets the next iterate
$y_{k+1}$ as the unique optimal
solution $y(\tilde x_k;M_k)$ of
the linearized prox subproblem 
\begin{equation}\label{eq:update}
y(\tilde x_k;M_k) :=\argmin\left\lbrace  \ell_f(x;\tx_k) + h(x) + \frac{M_k}2 \|x-\tx_k\|^2 : x \in \R^n \right\rbrace 
\end{equation}
where $\ell_f(x;\tx_k) := f(\tx_k) + \inner{\nabla f(\tx_k)}{x-\tx_k}$, the prox-center
$\tx_k$ is chosen as either the current iterate $y_k$ (as in unaccelerated algorithms)
or a convex combination of
$y_k$ and another auxiliary iterate $x_k$ (as in accelerated algorithms), and
$M_k$ is good upper curvature of
$f$ at $\tx_k$, i.e.,
$M_k>0$ and satisfies
\begin{equation}\label{ineq:descent}
{\cal C}(y(\tx_k;M_k);\tx_k) \le M_k
\end{equation}
where
\begin{equation}\label{eq:cal C}
{\cal C}(y;\tx) := \frac{2\left[ f(y)-\ell_f(y; \tx) \right] }{\|y-\tx\|^2}.
\end{equation}
Regardless of the choice of $\tx_k$, it is
well-known that
the smaller the sequence $\{M_k\}$ is, the faster
the convergence rate of the method becomes.
Hence, it is desirable to choose $M_k=\bar M_k$ where $\bar M_k$,
referred to  as the local curvature of $ f $ at $\tx_k$,
is the smallest value of $M_k$
satisfying \eqref{ineq:descent}.
However, since finding $\bar M_k$ is generally time-consuming,  alternative strategies that upper estimate
$\bar M_k$ are used. A common one 
is a backtracking procedure that initially sets
$M_k$ to be the maximum of all the observed curvatures ${\cal C}_1, \ldots,{\cal C}_{k-1}$ 
where ${\cal C}_i := {\cal C}(y_{i+1};\tx_i)$ for every $i \ge 1$. It then checks whether
$M_k$ is a good curvature of $f$ at $\tx_k$;
if so, it sets $y_{k+1}=y(\tx_k;M_k)$; otherwise,
it updates $ M_k \leftarrow \eta M_k $ for some parameter $ \eta>1 $, and then repeats
this same step again.
%
Such an approach has been used extensively in the literature dealing
with composite gradient methods both in the context of convex and nonconvex
SCO (N-SCO) problems (see for example \cite{beck2009fista,LanUniformly,liang2019fista,nesterov2012gradient})
and can be efficient particularly for those SCO instances
where a sharp upper bound $M$ on the smallest Lipschitz constant $\bar M$ of $\nabla f$ on $\dom h$
is not available.

This paper investigates an ACG variant for solving the N-SCO problem where $M_k$ is computed
as a positive multiple of the average of all observed curvatures up to the previous iteration. As opposed
to ACG variants based on the scheme outlined above as well as other ACG variants, AC-ACG always computes
a new step regardless of whether $M_k$ overestimates or underestimates ${\cal C}_k$. More specifically,
if $M_k$ overestimates ${\cal C}_k$ then a composite step as in \eqref{eq:update} is taken; otherwise,
$y_{k+1}$ is set to be a convex combination of $y_k$ and an auxiliary iterate $x_{k+1}$, which
is obtained by a resolvent evaluation of $h$. It is worth noting that both of these steps are used
in previous ACG variants but only one of them is used at a time.
The main result of the paper establishes a convergence rate for AC-ACG. More specifically, it states that $k$ iterations of the AC-ACG method generate a pair $ (y,v)$ satisfying
$ v \in \nabla f(y) + \partial h(y)$ and $\|v\|^2={\cal O}(M_k/k)$ where $M_k$
is as in the beginning of this paragraph.
Since $M_k$ is usually much smaller than $\bar M$ or even
$\bar M_k$, this convergence rate bound explains the efficiency  of AC-ACG to solve
both randomly generated and real-world problem instances of \eqref{eq:PenProb2Intro}
used in our numerical experiments.
Finally, it is shown that AC-ACG also has similar iteration-complexity as previous  ACG variants (e.g., \cite{nonconv_lan16,KongMeloMonteiro,jliang2018double,liang2019fista}).

{\bf Related works.} 
The first complexity analysis of an ACG algorithm for solving \eqref{eq:PenProb2Intro}
 under the assumption that $f$ is a nonconvex differentiable function whose gradient is
 Lipschitz continuous and that $h$ is a simple lower semicontinuous
convex function is established in the novel work \cite{nonconv_lan16}. 
Inspired by \cite{nonconv_lan16}, many papers have proposed other ACG variants for
solving (\ref{eq:PenProb2Intro}) under
the aforementioned  assumptions (see e.g.,\ \cite{Paquette2017, LanUniformly,liang2019fista}) 
or even under the relaxed assumption that
$ h $ is nonconvex (see e.g.,\ \cite{Li_Lin2015, Lin_Zhou_Liang_Varshney, Yao_et.al.}).
It is worth mentioning that: i)  in contrast to \cite{nonconv_lan16,liang2019fista}, the other works
deal with hybrid-type accelerated methods that resort to unaccelerated composite gradient steps
whenever a certain descent property is not satisfied; and ii) in contrast to the methods of
\cite{LanUniformly,Li_Lin2015,liang2019fista} that choose $M_k$ adaptively in a manner similar to that described in the second paragraph in Section \ref{sec:intro},
the methods in \cite{Paquette2017,nonconv_lan16,Lin_Zhou_Liang_Varshney,Yao_et.al.} works with a constant sequence $\{M_k\}$, namely, $M_k=M$ for some $M > \bar M$. 
Section \ref{sec:comparison} provides
a more detailed overview of ACG variants for solving both convex and nonconvex SCO problems
which includes most of the ones just mentioned.
%

Other approaches towards solving \eqref{eq:PenProb2Intro} use an inexact proximal point scheme where each prox subproblem
is constructed to be (possibly strongly) convex and hence efficiently solvable by a convex ACG variant.
Papers
\cite{carmon2018accelerated,KongMeloMonteiro,rohtua} propose a descent unaccelerated inexact 
proximal-type method, which works with a larger prox stepsize and hence has a better outer iteration-complexity than the approaches in the previous paragraph.
Paper \cite{jliang2018double} presents an accelerated inexact proximal point method that performs an accelerated step with a large prox stepsize in every outer iteration and requires a prox subproblem to be approximately solved by an ACG variant in the same way as in
the algorithms presented in  \cite{carmon2018accelerated,KongMeloMonteiro}.

{\bf Definitions and notations.} 
The set of real numbers is denoted by $\mathbb{R}$. The set of non-negative real numbers  and 
the set of positive real numbers are denoted by $\mathbb{R}_+$ and $\mathbb{R}_{++}$, respectively. Let $\mathbb{R}^n$ denote the standard $n$-dimensional Euclidean 
space with  inner product and norm denoted by $\left\langle \cdot,\cdot\right\rangle $
and $\|\cdot\|$, respectively. 
The Frobenius inner product and Frobenius norm in $ \R^{m\times n} $
are denoted by $ \inner{\cdot}{\cdot}_F $ and $ \|\cdot\|_F $, respectively.
The set of real $ n\times n $ symmetric matrices is denoted by $ {\cal S}^n $, and we define $ {\cal S}^n_{+} $ to be the subset of $ {\cal S}^n $ consisting of the positive semidefinite matrices.
The indicator function $I_S$ of a set
$ S\subset \R^n $ is defined as $ I_S (z) =0 $ for every $ z\in S, $ and
$ I_S (z) =\infty $, otherwise. 
The cardinality of a finite set $ {\cal A} $ is denoted by $ |{\cal A}| $. 
Let ${\cal O}_1(\cdot)$ denote ${\cal O}(\cdot + 1)$ where ${\cal O}$ is the big O notation.

Let $\psi: \mathbb{R}^n\rightarrow (-\infty,+\infty]$ be given. The effective domain of $\psi$ is denoted by
$\dom \psi:=\{x \in \mathbb{R}^n: \psi (x) <\infty\}$ and $\psi$ is proper if $\dom \psi \ne \emptyset$.
Moreover, a proper function $\psi: \mathbb{R}^n\rightarrow (-\infty,+\infty]$ is said to be $\mu$-strongly convex for some $\mu \ge 0$ if
$$
\psi(\alpha z+(1-\alpha) u)\leq \alpha \psi(z)+(1-\alpha)\psi(u) - \frac{\alpha(1-\alpha) \mu}{2}\|z-u\|^2
$$
for every $z, u \in \dom \psi$ and $\alpha \in [0,1]$.
If $\psi$ is differentiable at $\bar z \in \mathbb{R}^n$, then its affine   approximation $\ell_\psi(\cdot;\bar z)$ at $\bar z$ is defined as
\[
\ell_\psi(z;\bar z) :=  \psi(\bar z) + \inner{\nabla \psi(\bar z)}{z-\bar z} \quad \forall  z \in \mathbb{R}^n.
\]
The subdifferential of $\psi$ at $z \in \R^n$ is denoted by $\partial \psi (z)$.
The set of all proper lower semi-continuous convex functions $\psi:\mathbb{R}^n\rightarrow (-\infty,+\infty]$  is denoted by $\bConv{n}$.

{\bf Organization of the paper.} 
Section~\ref{sec:alg+cmplx} describes the N-SCO problem and the assumptions made on it.
It also presents the AC-ACG method for solving the N-SCO problem
and describes the main result of the paper, which
establishes a convergence rate bound for AC-ACG in terms of the average of observed curvatures.
Section \ref{sec:comparison} contains three subsections. The first subsection reviews three ACG variants for solving convex SCO (C-SCO) problems. The second (resp. third) one reviews pure (resp. hybrid) ACG variants for solving N-SCO problems.
Section~\ref{sec:proof} provides the proof of the main result stated in Section \ref{sec:alg+cmplx}.
Section~\ref{sec:numerics} presents computational results illustrating the efficiency of the AC-ACG method.
Section \ref{sec:conclusion} presents some concluding remarks.
Finally, the appendix contains a technical result.

\section{The AC-ACG method for solving the N-SCO problem}\label{sec:alg+cmplx}

This section presents the main algorithm studied in this paper, namely, an ACG method based on a sequence of average curvatures, and derives a convergence rate for it expressed in terms of this sequence.
More specifically,
it describes the N-SCO problem and the assumptions made on it,
presents the AC-ACG method and states the main result of the paper, i.e., 
the convergence rate of the AC-ACG method.


The problem of interest in this paper is the N-SCO problem \eqref{eq:PenProb2Intro},
where the following conditions are assumed to hold:
\begin{itemize}
	\item[(A1)] $h \in \bConv{n}$;  
	\item[(A2)] $f$ is a nonconvex differentiable function on $ \dom h$ and there exist scalars
	$m \ge 0$, $ M \ge 0 $ such that
for every $u,u' \in \dom h$,
	\beq\label{ineq:curv}
	-\frac {m}2\|u-u'\|^2\le f(u)-\ell_f(u;u'),  \qquad 
\|\nabla f(u)- \nabla f(u') \| \le M\|u-u'\|;
	\eeq
	\item[(A3)]  the diameter $D  :=  \sup \{  \| u-u' \| : u, u' \in \dom h\}$ is bounded.
\end{itemize}

Throughout the paper, we let $ \bar m $ (resp., $\bar M$) denote the smallest scalar
$m \ge 0$ (resp., $M \ge 0$)
satisfying the first (resp., second) inequality in
\eqref{ineq:curv}.

We now make some remarks about the above assumptions.
First, the set of optimal solutions $ X_* $ is nonempty and compact in view of (A1)-(A3).
Second, the second inequality in \eqref{ineq:curv} implies
\begin{equation}\label{ineq:upper}
- \frac {M}2\|u-u'\|^2 \le  f(u)-\ell_f(u;u')\le\frac {M}2\|u-u'\|^2 \quad \forall u,u'\in \dom h.
\end{equation}
Third, the last remark together with the fact that $f$ is nonconvex on $\dom h$ due to assumption
(A2) implies that
$0 < \bar m\le \bar M$. 
Fourth, assumption (A3) is
used in the proofs of Lemma \ref{lem:tech1}(b) and Lemma \ref{lem:basic}(b).

A necessary condition for $\hat y$ to be a local minimum of \eqref{eq:PenProb2Intro} is that
 $ 0\in \nabla f(\hat y) + \partial h(\hat y) $, i.e,  $ \hat y $ be a stationary point of \eqref{eq:PenProb2Intro}.
More generally, given a tolerance $ \hat \rho>0 $, a pair $ (\hat y, \hat v) $ is called a $ \hat \rho $-approximate
stationary pair of $ \eqref{eq:PenProb2Intro} $ if
\begin{equation}\label{incl:v}
\hat v \in \nabla f(\hat y) + \partial h(\hat y), \quad \|\hat v\|\le \hat \rho.
\end{equation}

We are now ready to state the
AC-ACG method, which
stops when a $ \hat \rho $-approximate
stationary pair of $ \eqref{eq:PenProb2Intro} $ is computed. AC-ACG requires as
input a scalar
$M \ge \bar M$ where $\bar M$ is defined in
the paragraph following (A3).

%

\noindent\rule[0.5ex]{1\columnwidth}{1pt}

Average Curvature - Accelerated Composite Gradient (AC-ACG)

\noindent\rule[0.5ex]{1\columnwidth}{1pt}

\begin{itemize}
	\item[0.] Let a parameter $ \gamma \in(0,1)$,
a scalar $M\ge \bar M$, a tolerance $ \hat \rho>0 $  and an initial point $y_0 \in \dom h$ be given and set $A_0 =0$, $x_0=y_0$, $M_0=\gamma M$, $k=0$ and
\begin{equation}\label{eq:alpha}
\alpha = \frac{0.9}{8}\left( 1+\frac{1}{0.9\gamma}\right)^{-1};
\end{equation}
	\item[1.] compute
	\begin{equation}\label{eq:aktx}
	a_k = \frac{1+ \sqrt{1 + 4 M_{k}A_k}}{2M_{k}},\quad
	A_{k+1} = A_k + a_k, \quad
	\tx_k  =  \frac{A_ky_k+a_kx_k}{A_{k+1}};
	\end{equation}
	\item[2.] set $y_{k+1}^g=y(\tx_k;M_k)$ where
	$y(\cdot;\cdot)$ is as in \eqref{eq:update} and compute
	\begin{align}
	x_{k+1} &= \underset{u\in \R^n}{\mbox{argmin}} \left\{ a_k \left[ \ell_f(u; \tilde{x}_k) + h(u) \right] + \frac{1}{2} \| u - x_k \|^2 \right\},\label{eq:x+}\\ 
	v_{k+1} &=  M_k(\tilde{x}_k - y^g_{k+1}) + \nabla f(y^g_{k+1}) - \nabla f(\tilde{x}_k); \label{eq:vk}
	\end{align}
	\item[3.]
	if $\|v_{k+1} \| \le \hat \rho$ then output  $(\hat y,\hat v)=(y^g_{k+1},v_{k+1})$ and {\bf stop}; otherwise, compute
	\begin{align}
 C_k &=  \max\left\lbrace 
{\cal C}(y_{k+1}^g;\tx_k),\frac{\|\nabla f(y^g_{k+1}) - \nabla f(\tilde{x}_k)\|}{\|y^g_{k+1}-\tx_k\|}
	\right\rbrace,  \label{eq:C} \\
	C^{avg}_k&=\frac{1}{k+1} \sum_{j=0}^{k} C_j, \label{eq:Cave} \\
	M_{k+1} &= \max \left\{ \frac1\alpha C^{avg}_k, \gamma M \right \} \label{eq:M}
	\end{align}
	where ${\cal C}(\cdot;\cdot)$ is as in \eqref{eq:cal C};
	\item[4.] set
	\begin{equation}
	y_{k+1} = \left\{ \begin{array}{cc}  
	y_{k+1}^b:=\frac{A_k y_k + a_k x_{k+1}}{A_{k+1}}, & \mbox{if $C_k>0.9M_k$};  \\ [.1in]
    y^g_{k+1}, & \mbox{otherwise} 
    \end{array} \right. \label{eq:ty}
	\end{equation}
and $k \leftarrow k+1$, and go to step 1.
\end{itemize}

\noindent\rule[0.5ex]{1\columnwidth}{1pt}

We add a few observations about the AC-ACG method.
First, the first two identities in \eqref{eq:aktx} imply that
\begin{equation}\label{eq:rel}
A_{k+1} = M_k a_k^2.
\end{equation}
Second, the AC-ACG method evaluates two gradients of $f$ and exactly two resolvents of $h$,
(i.e., an evaluation of $(I+\lambda \partial h)^{-1}(\cdot)$ for some $\lambda>0$) per iteration,
namely, one in \eqref{eq:update} and the other one in \eqref{eq:x+}.
Third, Theorem \ref{thm:main} below guarantees that AC-ACG always terminates and outputs
a $\hat \rho$-approximate solution $(\hat y, \hat v)$ (see step 3).
Fourth, $C_k$ is the most recent observed curvature, $C_k^{avg}$ is the average of all observed curvatures obtained
so far and $M_{k+1}$ is a modified average curvature that will be used in the
next iteration to compute $y_{k+2}^g$. Fifth, the observed curvature $C_k$ used here is
different from the one mentioned in the Introduction (see \eqref{ineq:descent})
and it is more suitable
for our theoretical analysis. 
Sixth, every iteration starts with a triple $(A_k,x_k,y_k)$ and
obtains the next one $(A_{k+1},x_{k+1},y_{k+1})$ as in \eqref{eq:aktx}, \eqref{eq:x+} and \eqref{eq:update}.
The iterate $y_{k+1}$ is chosen to be either $y^g_{k+1}=y(\tx_k;M_k)$ obtained in \eqref{eq:update}
or the convex combination $y^b_{k+1}$ defined in \eqref{eq:ty} depending on whether
the current curvature $C_k$ is smaller than or equal to a multiple (e.g., $0.9$) of
the modified average curvature $M_k$ or not, respectively.
Seventh, in the iterations for
which $C_k \le 0.9M_k$ (called the good ones), 
$ M_k $ is clearly a good upper curvature of $f$ at
$\tx_k$ in view of the definitions of $ C_k $ and $y_{k+1}^g$
in \eqref{eq:C} and step 2 of AC-ACG, respectively,
and the definition of a good curvature in \eqref{ineq:descent}.
Thus, assuming that the
frequency of good iterations is relatively high,
it is reasonable to expect that the smaller the sequence
$\{M_k\}$ is,
the faster the convergence rate of AC-ACG will be
(see the discussion after \eqref{ineq:descent} in the Introduction).
Eighth, it follows as a consequence of
the results of Section \ref{sec:proof}
that the number of good iterations is
relatively large (see Lemma \ref{lem:card}) and that the overall effect of the bad ones are nicely under control (see Lemma \ref{lem:sum}).
Moreover, Theorem \ref{thm:main} below states
that the convergence rate of AC-ACG
is directly proportional to $\sqrt{M_k}$ in that
$\min \{\|v_i\| : i \le k \} = {\cal O}(\sqrt{M_k}/\sqrt{k})$.

We now discuss the likelihood of $M_{k+1}$,
or equivalently, $\gamma_{k+1}:= M_{k+1}/ M$, being small. First observe that \eqref{eq:M} implies that
$\gamma_{k+1}\ge \gamma$. Hence, let us examine the situation
in which $\gamma_{k+1}= \gamma$, i.e., $\gamma_{k+1}$
reaches its lowest possible value for a fixed $\gamma \in (0,1)$.
Clearly, it follows from \eqref{eq:M} that
$\gamma_{k+1}= \gamma$ if and only if
\beq \label{ineq:gamma'}
\frac{C_k^{avg}}M \le \alpha\gamma.
\eeq
Moreover, in view of \eqref{eq:alpha} and the fact
that $\gamma<1$, it follows that
$\alpha={\Theta}(\gamma)$, and hence 
\eqref{ineq:gamma'} implies that
$C_k^{avg}/M = {\cal O}(\gamma^2)$.
In conclusion, under the restrictive choice of $\alpha$
in \eqref{eq:alpha}, $\gamma_{k+1}=\gamma$ can only
happen when 
the computed average curvature ratio $C_k^{avg}/M$ is
${\cal O}(\gamma^2)$.
However, choice \eqref{eq:alpha} for 
$\alpha$ is too conservative in practice.
Indeed, it follows from
the proof of Lemma \ref{lem:card} and the
arguments in the paragraph following it that
in practice $\alpha \in (0,1)$ can be chosen as $\Theta(1)$
instead of $\Theta(\gamma)$ as above. Clearly, with such
a choice of $\alpha$, \eqref{ineq:gamma'} implies that
the ratio $C_k^{avg}/M$ is ${\cal O}(\gamma)$
instead of ${\cal O}(\gamma^2)$ as above.
In summary, if $\gamma \in (0,1)$ is relatively small
and $\alpha$ is chosen as $(0,1) \ni \alpha = \Theta(1)$ instead of \eqref{eq:alpha},
then the chances of having $\gamma_{k+1}=\gamma$
increases.
In view of the aforementioned observation,
the two AC-ACG variants which are
computationally profiled
in
Section \ref{sec:numerics}
relax the choice of $\alpha$
from \eqref{eq:alpha} to one satisfying
$(0,1) \ni \alpha = \Theta(1)$.

We now state the main result of the paper which describes
how fast one of the iterates $y_1^g,\ldots,y_k^g$ approaches the
stationary condition $0 \in \nabla f(y)+\partial h(y)$. A remarkable feature of its convergence
rate bound is that it is expressed in terms of $M_k$ rather than
a scalar $M \ge \bar M$.

\begin{theorem}\label{thm:main}
The following statements hold:
\begin{itemize}

\item[(a)]
for every $k \ge 1 $, we have
$v_k \in \nabla f(y_{k}^g)+\partial h(y_{k}^g)$;
\item[(b)]
	for every $ k\ge 12 $,  we have
	\[
	\min_{1 \le i \le k}\|v_{i}\|^2= {\cal O}
	\left( 
	\frac{ M_k^2D^2}{\gamma k^2} +
	\frac{\theta_k\bar mM_kD^2}{k} \right)
	\]
where 
\begin{equation}\label{def:theta}
\theta_{k}:= \max \left\{\frac{M_k}{M_i}  : 0 \le i \le k \right\} \ge 1.
\end{equation}
\end{itemize}
\end{theorem}

We now make two remarks about Theorem \ref{thm:main}.
First, it immediately leads to a worst-case iteration-complexity bound as follows.
In view of the second inequality in \eqref{ineq:curv},
the second inequality in \eqref{ineq:upper}, the definition of $ \bar M $ in the paragraph following (A3), and relation \eqref{eq:C}, it follows that
for every $k \ge 0$,
$C_k \le \bar M$, and
hence that $C_k^{avg} \le \bar M$
in view of \eqref{eq:Cave}.
The latter inequality, \eqref{eq:M}, and the fact that
$\alpha = {\Theta}(\gamma)$ (see the line following
\eqref{ineq:gamma'}),
then imply that
$M/M_{k+1}\le 1/\gamma $ and
\begin{equation}\label{ineq:estimates}
\frac{M_{k+1}}M = {\cal O}\left( \frac{\bar M}{\alpha M} + \gamma  \right) = {\cal O}\left( \frac{\bar M}{\gamma M} + \gamma  \right)    
\end{equation}
for every $k \ge 0$.
These two estimates and the definition of $ \theta_k $ in \eqref{def:theta} then 
imply that,
for some $ i\le k $, we have
\[
\theta_k
=\frac{M_k}{M} \frac{M}{M_i} ={\cal O}\left( \left( \frac{\bar M}{\gamma M} + \gamma\right) \frac{1}{\gamma} \right) 
={\cal O}\left( \frac{\bar M}{\gamma^2M} + 1 \right).
\]
Moreover, it follows from Theorem \ref{thm:main}(b) that the iteration-complexity for  AC-ACG to obtain a $\hat \rho$-approximate stationary pair $(\hat y,\hat v)$ is
\[
{\cal O}_1\left( \frac{ M_kD}{\gamma^{1/2} \hat \rho} +
\frac{\theta_k\bar mM_kD^2}{\hat \rho^2} \right)
={\cal O}_1\left( M_k\left( \frac{ D}{\gamma^{1/2} \hat \rho} +
\theta_k\frac{\bar mD^2}{\hat \rho^2} \right)\right)
\]
which, in view of
\eqref{ineq:estimates}, the above estimate on $ \theta_k $, and the facts that $\gamma<1$ and $M\ge \bar M$ (see step 0 of AC-ACG),
is bounded by
\begin{equation}\label{eq:bound}
{\cal O}_1
\left(
\left[\frac{ D}{\gamma^{1/2}\hat \rho} + \left( \frac{\bar M}{\gamma^2M} +1\right)
\frac{\bar m D^2}{\hat \rho^2} \right]
\left( \frac{\bar M}{\gamma} + \gamma M \right)  \right)
= {\cal O}_1\left( \frac{MD}{\gamma^{3/2}\hat \rho} + \frac{\bar m M D^2}{\gamma^3\hat \rho^2}\right ).
\end{equation}
Hence, for small values of $\gamma$, the worst-case iteration-complexity of AC-ACG
is high but, if $\gamma$ is viewed as a constant, i.e., $1/\gamma = {\cal O}(1)$,
then the above complexity is as good as any other ACG method found in the literature 
for solving the N-SCO problem as long as the second term in \eqref{eq:bound} is the dominant one.
In particular, in terms of $\hat \rho$ only,
its worst-case iteration-complexity for solving an N-SCO problem is ${\cal O}(1/{\hat \rho}^2)$, which is
identical to that of any other
known ACG method (see e.g.,\ \cite{nonconv_lan16,KongMeloMonteiro,jliang2018double,liang2019fista}).

Second, the dependence of the worst-case iteration-complexity \eqref{eq:bound} on
$\gamma$ is not good because it is obtained using
the conservative estimate \eqref{ineq:estimates}.
We will now examine the iteration-complexity bound
under the assumption that
$\gamma_{k+1}=M_{k+1}/ M=\gamma $, or equivalently, \eqref{ineq:gamma'} holds, for every $k \ge 0$.
In this case, $\theta_k = 1$ for every $k \ge 0$
and hence the convergence rate bound
in Theorem \ref{thm:main}(b) yields the
iteration-complexity bound
\[
{\cal O}_1
\left( 
  \frac{\gamma^{1/2} MD}{\hat \rho} +
 \frac{\gamma\bar mMD^2}{\hat \rho^2} \right)
\]
for AC-ACG,
which improves as $\gamma$ decreases. This contrasts with bound
\eqref{eq:bound}, which becomes worse as $\gamma$ decreases.

\section{Comparison with other accelerated type methods}\label{sec:comparison}
This section gives a brief overview of existing ACG methods for solving convex and nonconvex SCO problems. It contains three subsections. The first subsection reviews three ACG variants for solving C-SCO problems. The second one discusses pure ACG variants for solving N-SCO problems, i.e., ACG variants which perform only accelerated steps similar to the ones of the variants
of  the first subsection.
The third one discusses hybrid ACG variants which, in addition to accelerated composite gradient steps, may also perform unaccelerated  ones.

\subsection{Review of convex ACG methods}\label{Subsec:review}
This subsection reviews three ACG variants for solving C-SCO problems,
i.e., SCO problems of the form \eqref{eq:PenProb2Intro} where (A1)-(A3) hold with $m=0$, and
hence $ f $ is convex.
All the ACG methods reviewed here are described in
terms of the notation introduced in the AC-ACG method
or the ACG framework described below. This approach
has the advantage that all the ACG methods are viewed under the
same notation and hence their similarities/differences
become more apparent.

The accelerated gradient method for solving unconstrained
C-SCO problems (i.e., \eqref{eq:PenProb2Intro}
with $h=0$) were originally developed by Nesterov in his celebrated work \cite{ag_nesterov83}. Subsequently, several variants of his method (see for example \cite{auslender2006interior,beck2009fista,lan2011primal,MonteiroSvaiterAcceleration,nesterov1998,Nest05-1,nesterov2012gradient,tseng2008accmet}) have been developed for solving C-SCO problems.

Before reviewing  ACG variants for solving
C-SCO,  we first describe
a common ACG framework underlying them.


\noindent\rule[0.5ex]{1\columnwidth}{1pt}

ACG framework

\noindent\rule[0.5ex]{1\columnwidth}{1pt}

\begin{algorithmic}
	\STATE 0. Let an initial point $ y_0 \in \dom h$ be given, and set $ x_0=y_0 $, $ A_0=0 $ and $ k=0 $;
	\STATE 1. compute $a_k$, $A_{k+1}$ and $\tx_k$ as
	in \eqref{eq:aktx};
	\STATE 2. compute $x_{k+1}$ and $y_{k+1}$ using one
	of the rules listed below;
	\STATE 3. set $k \leftarrow k+1$, and go to step 1.
\end{algorithmic}
\noindent\rule[0.5ex]{1\columnwidth}{1pt}

We will now describe three possible rules for computing
the iterates
$x_{k+1}$ and $y_{k+1}$ in step 2 of the above framework.

\begin{itemize}
    \item[i)] (FISTA rule) 
    This rule sets $y^a_{k+1} = y(\tx_k;M_k)$
    where $y(\tx_k;M_k)$ is defined in \eqref{eq:update}
    and $M_k>0$ is a good upper curvature of $f$ at $\tx_k$,
    then chooses $y_{k+1}$ to be any point
    satisfying
    $ \phi(y_{k+1}) \le \phi(y^a_{k+1})$
    and computes $x_{k+1}$ as
    \begin{equation}\label{FISTA x}
	x_{k+1}
	=y_{k+1}^a+\frac{A_k}{a_k}\left(y_{k+1}^a-y_k\right).
    \end{equation}
    FISTA rule with $y_{k+1}=y^a_{k+1}$ was first introduced by Nesterov
    when $h$ is the indicator function
    of a nonempty closed convex set
    (see for example ``Constant Step Scheme, III" on pages 83-84 of \cite{nesterov1998}
    or ``Constant Step Scheme, II. Simple sets" on page 90 of \cite{nesterov2004})
    and was later extended to general composite
    closed convex functions in \cite{beck2009fast,beck2009fista}.
    \item[ii)] (AT rule)
    This rule computes
    $x_{k+1}$ as \eqref{eq:x+}
    and chooses $y_{k+1}$
    to be any point satisfying $ \phi(y_{k+1}) \le \phi(y^a_{k+1}) $ where
    \begin{equation}\label{non-FISTA y}
    y_{k+1}^a=\frac{A_ky_k + a_kx_{k+1}}{A_{k+1}}.
    \end{equation}
    This rule with $y_{k+1}=y_{k+1}^a$ was introduced by
    Auslender and Teboulle in \cite{auslender2006interior},
    which explains the  name ``AT"
    adopted here.
    \item[iii)] (LLM rule) 
    This rule sets
    $y_{k+1}$ as in the FISTA rule and 
    and $x_{k+1}$ as in the AT rule.
    LLM rule was introduced by Lu, Lan and Monteiro in \cite{lan2011primal}, which explains the name ``LLM" adopted here.
    
\end{itemize}

We now make a few remarks on the three ACG variants based on the above three rules. First, the ACG variant based on the LLM rule performs two resolvent evaluations of $ h $ per iteration, while the variants based on the AT and FISTA rules perform exactly one resolvent evaluation.
Second, two popular choices of
an upper curvature sequence $\{M_k\}$
are as follows:
1) for some $M \ge \bar M$, $M_k=M$ for every $k\ge0$; and
2) for every $k \ge 0$,
$ M_k $ is computed by a backtracking procedure
such as the one outlined in the second paragraph
of Section \ref{sec:intro}.
While \cite{auslender2006interior,lan2011primal,nesterov1998} consider only the first choice, \cite{beck2009fista,nesterov2012gradient}
analyze the FISTA variant for both  choices of $\{M_k\}$.
Third, the AC-ACG method studied in this paper uses the LLM rule and works with a sequence
$ \{M_k\} $ such that $M_k$ is not necessarily
a good upper curvature of $f$ at $\tx_k$.

We now comment on the monotonicity of the three
aforementioned ACG variants. The three ACG variants
based on the identity $ y_{k+1}=y_{k+1}^a $
are not necessarily monotone (i.e., it satisfies
$\phi(y_{k+1}) \le \phi(y_k)$ for every $k\ge 0$),
even if every $ M_k $ is a good upper curvature of $f$ at $\tx_k$.
However, they can be made monotone
by invoking an idea introduced in \cite{Nest05-1}
which sets 
$y_{k+1} =\argmin\{\phi(y):y\in\{y_k,y_{k+1}^a\}\}$,
where $y^{a}_{k+1}$ is as described in each of the rules
above.
Another alternative way of forcing monotonicity,
which requires an extra resolvent evaluation of $h$,
is to choose $y_{k+1}$ as
\begin{equation}\label{y+}
y_{k+1} =
\argmin\{\phi(y):y\in\{y_k,y_{k+1}^a,y_{k+1}^{na}\}\}
\end{equation}
where $ y_{k+1}^{na} =y(y_k;M_k^{na})$ and $ M_k^{na} $ is a good upper curvature of $f$ at $y_k$.
We remark that $y_k$ can actually be removed from
the right hand side of \eqref{y+}.
This is due to the fact that $M_k^{na}$ being a good upper curvature of $f$
at $y_k$ implies that
$\phi(y^{na}_{k+1}) \le \phi(y_k)$
in view of Lemma \ref{lem:descent} in the Appendix with $(M_k,\tx_k,y_{k+1})=(M_k^{na},y_k,y_{k+1}^{na})$.

\subsection{Pure accelerated variants}\label{subsec:pure}

This subsection discusses pure ACG variants for solving the N-SCO problem \eqref{eq:PenProb2Intro}.
More specifically, we discuss three methods, namely:
the AG method proposed in \cite{nonconv_lan16}, the NC-FISTA of \cite{liang2019fista},
and its adaptive variant ADAP-NC-FISTA also described in \cite{liang2019fista}.
The iteration-complexity of
all three methods are analyzed under the assumption that
$ \dom h $ is bounded, but in practice
all three methods can successfully solve many problems with unbounded $ \dom h $.


AG is a direct extension
of the ACG variant, based on the LLM rule and the constant choice of $M_k$, to
the N-SCO context.
Clearly, AG performs two resolvent evaluations of $ h $ per iteration.


NC-FISTA requires as input a pair $(M,m)$ such that $M > \bar M$ and $M\ge m \ge \bar m$.
It is an extension of the version of FISTA with $y_{k+1}=y^a_{k+1}$ from the C-SCO to the N-SCO context, and it reduces to the latter one when $m=\bar m=0$.
More specifically, NC-FISTA sets $ y_{k+1}=y(\tx_k;M_k)$ where
$ M_k=M+\kappa_0m/(Ma_k) $, and computes $ x_{k+1} $ as in \eqref{FISTA x} with $ A_k/a_k $ replaced by
$(\kappa_0 m/M + 1)^{-1}(A_k/a_k)$
where
$ \kappa_0 $ is a positive universal constant.
In contrast to an iteration of the AG method, every iteration of NC-FISTA performs exactly one resolvent evaluation of $ h $.

One drawback of NC-FISTA is its required
input pair $ (M,m) $, which is usually
hard to obtain or is often poorly estimated.
On the other hand, ADAP-NC-FISTA remedies this drawback in that it only requires as input an
arbitrary initial pair $(M_0,m_0) $ such
that $M_0\ge m_0>0$, which is dynamically updated by means of two separate backtracking search procedures.

\subsection{Hybrid accelerated variants}\label{subsec:hybrid}

This subsection discusses hybrid ACG variants for solving the N-SCO problem \eqref{eq:PenProb2Intro}.
More specifically, we discuss three methods, namely: a  non-monotone variant
as well as a monotone one both described in
 \cite{Li_Lin2015}, which we refer to as
NM-APG and M-APG, respectively,
and UPFAG proposed in \cite{LanUniformly}.
To the best of our knowledge, the convergence of these
hybrid ACG variants is guaranteed due to the possibility of performing an extra
unaccelerated composite gradient step. Whether their convergence can be shown without this optional step
is an open question even for the case in which $\dom h$ is bounded.




M-APG is exactly the instance of
the ACG variant based on the FISTA rule
in which $y_{k+1}$ is computed by means of
\eqref{y+} which, as already mentioned above,
guarantees its monotonicity property due to the
fact that $M_k^{na}$ is chosen as a good upper
curvature of $f$ at $y_k$.
NM-APG is a variant of M-APG, which
either sets $y_{k+1}=y_{k+1}^a$ or computes $y_{k+1}$ as in \eqref{y+} depending on whether or not,
respectively,
$y^a_{k+1}$ satisfies a key inequality, which
ensures convergence of the method
but not necessarily
its monotonicity.

UPFAG is an ACG variant based on the AT rule in which the next iterate $ y_{k+1} $ is chosen as in \eqref{y+} except that $(M_k^a,M_k^{na})$ is
computed by line searches so that
$M_k^a$ closely approximates a
good curvature of $f$ at
$\tx_k$ and $M_k^{na}$  satisfies
a relaxed version of the descent condition
\eqref{ineq:implication} with $(M_k,\tx_k,y_{k+1})=(M_k^{na},y_k,y_{k+1}^{na})$.

\section{Proof of Theorem \ref{thm:main}}\label{sec:proof}

This section presents  the proof of Theorem \ref{thm:main}. 
We start with the following technical result, which assumes that all sequences start with $k=0$.

%

\begin{lemma}\label{lem:tech1}
The following statements hold:
\begin{itemize}
\item[(a)]
the sequences $\{x_k\}$, $\{y_k\}$, $\{y_{k+1}^g\}$, $\{y_{k+1}^b\}$ and
$\{\tx_k\}$ are all contained in $\dom h$;
\item[(b)] for every $u \in \dom h$ and $k \ge 0$, we have
	\[
	A_k \| y_k - \tilde{x}_k \|^2 + a_k \| u - \tilde{x}_k\|^2 \le a_k D^2;
	\]
	\item[(c)] for every $k \ge 0$, $ C_k\le \bar M $ and $ F_k\le \bar M $, where
	\begin{equation}\label{eq:Fk}
	F_k:= {\cal C}(y_{k+1};\tx_k)
	\end{equation}
	and $ {\cal C}(\cdot;\cdot) $ is defined in \eqref{eq:cal C};
	\item[(d)] 
	for every $ k\ge0 $, 
	we have
	\begin{equation}\label{eq:optcond}
	v_{k+1}\in \nabla f(y^g_{k+1})+\partial h(y^g_{k+1}), \quad \|v_{k+1}\|\le (M_k+C_k)\|y^g_{k+1}-\tx_k\|.
	\end{equation}
\end{itemize}
\end{lemma}
\begin{proof}
	(a) The sequences $ \{x_k\}$ and $ \{y_{k+1}^g\} $ are contained in $ \dom h $ in view of \eqref{eq:x+},
	 \eqref{eq:update} and step 0 of AC-ACG.
	 Hence, using step 0 of AC-ACG again,  \eqref{eq:ty} and the convexity of $\dom h$, we easily
	 see by induction that $\{y_k\}$ and $\{y^b_{k+1}\}$ are contained in
	 $ \dom h$. Finally, $\{\tx_k\} \subset \dom h$ follows from the third identity in \eqref{eq:aktx} and the convexity of $\dom h$.
	
	(b) Let $u \in \dom h$ and $k \ge 0$ be given. First note that for every $ A, a \in \R_+ $ and $ x, y\in \R^n $, we have
	\[
	A\|y\|^2+a\|x\|^2=(A+a)\left\| \frac{Ay+ax}{A+a}\right\|^2 + \frac{Aa}{A+a}\|y-x\|^2.
	\]
	Applying the above identity with $ A=A_k $, $ a=a_k $, $ y=y_k-\tx_k $ and $ x = u-\tx_k $, and using both the second and the third identities in \eqref{eq:aktx}, we have
	\begin{align*}
	A_k \| y_k - \tilde{x}_k \|^2 + a_k \| u - \tilde{x}_k\|^2  
	 &= A_{k+1}\left \| \frac{A_k y_k +a_k u}{A_{k+1}} -\tx_k\right \|^2 + \frac{A_ka_k}{A_{k+1}}\|y_k-u\|^2 \\
	 &=   \frac{a_k}{A_{k+1}}\left( a_k\|u-x_k\|^2 + A_k\|u-y_k\|^2\right) \le a_k  D^2
	\end{align*}
	where the inequality follows from  Lemma \ref{lem:tech1}(a), the assumption that $u \in \dom h$,
	the definition of $ D $ in (A3), and the second equality in \eqref{eq:aktx}.

	(c) The conclusion follows from definitions of $ C_k $, $ F_k $ and $ {\cal C}(\cdot;\cdot) $ in \eqref{eq:C}, \eqref{eq:Fk} and \eqref{eq:cal C}, respectively, and the fact that $ \bar M $ satisfies both the second inequality in \eqref{ineq:curv} and \eqref{ineq:upper}.
	
	(d) The inclusion in \eqref{eq:optcond} follows from the fact $y_{k+1}=y(\tx_k;M_k)$, the optimality condition of \eqref{eq:update} and the definition
	of $ v_{k+1} $ in \eqref{eq:vk}. Moreover, the inequality in \eqref{eq:optcond} follows from
	 definitions of $ C_k $ in \eqref{eq:C} and $v_{k+1}$, and the triangle inequality.
%
\end{proof}

The next result provides an important recursive formula involving
a certain potential function $\eta_k$ and the quantity $\|y_{k+1}-\tx_k\|$ that will later be related to
the residual vector $\|v_{k+1}\|$ (see the proof of Lemma \ref{lem:basic}(a)).

\begin{lemma}\label{lem:tech2}
	For every $ k\ge 0 $ and $u \in \dom h$, we have
	\[
	\frac{ M_k-F_k}{2}A_{k+1}\|y_{k+1}-\tx_k\|^2
	\le \eta_k(u)- \eta_{k+1}(u)+\frac{1}{2}  \bar m a_k D^2
	\]
	where $M_k$ and $ F_k $ are as in \eqref{eq:M} and \eqref{eq:Fk}, respectively, and
	\begin{equation} \label{eq:eta}
	\eta_k(u):=A_k(\phi(y_k)-\phi(u))+\frac12\|u-x_k\|^2.
	\end{equation}
\end{lemma}
\begin{proof}
Let $k \ge 0$ and $u \in \dom h$ be given and define
$ \gamma_k(u):=\ell_f(u;\tx_k)+h(u) $.
	Using the fact $ x_{k+1} $ is an optimal solution of \eqref{eq:x+} and $\gamma_k$ is a convex function, the second and third identities in \eqref{eq:aktx}, and
relations \eqref{eq:ty} and \eqref{eq:rel}, we conclude that
	\begin{align*}
	A_k\gamma_k(y_k)+a_k\gamma_k(u)+\frac{1}{2}\|u-x_k\|^2-\frac{1}{2}\|u-x_{k+1}\|^2 
	\ge& A_k\gamma_k(y_k)+a_k\gamma_k(x_{k+1})+\frac{1}{2}\|x_{k+1}-x_k\|^2 \nn \\
	\ge& A_{k+1}\gamma_k(y_{k+1}^b)+\frac{1}{2}\frac{A_{k+1}^2}{a_k^2}\|y_{k+1}^b-\tx_k\|^2 \nn \\
	=& A_{k+1}\left[ \gamma_k(y_{k+1}^b)+\frac{ M_k}{2}\|y_{k+1}^b-\tx_k\|^2\right].
	\end{align*}
Moreover, relations \eqref{eq:update}, \eqref{eq:ty} and \eqref{eq:Fk}, and the fact that
$\{y_k^b\} \subset \dom h$ imply that
\[
\gamma_k(y_{k+1}^b)+\frac{ M_k}{2}\|y_{k+1}^b-\tx_k\|^2
\ge \gamma_k(y_{k+1})+\frac{ M_k}{2}\|y_{k+1}-\tx_k\|^2 = \phi(y_{k+1})+\frac{ M_k-F_k}{2}\|y_{k+1}-\tx_k\|^2.
\]
	Using the above two inequalities, the definition of $ \eta_k $ in \eqref{eq:eta} and the first inequality in \eqref{ineq:curv}, we easily see that
	\[
	\begin{aligned}
	\frac{ M_k-F_k}{2}A_{k+1}\|y_{k+1}-\tx_k\|^2-\eta_k(u)+ \eta_{k+1}(u)
	&\le A_k(\gamma_k(y_k)-\phi(y_k))+a_k(\gamma_k(u)-\phi(u))\\
	&\le \frac{\bar m}{2}\left( A_k \| y_k - \tx_k \|^2 + a_k\| u - \tx_k\|^2\right),
	\end{aligned}
	\]
	which, together with Lemma \ref{lem:tech1}(b), then immediately implies the lemma.
\end{proof}

For the purpose of stating the next results, we define the set of good and bad iterations as
\begin{equation}\label{def:G}
{\cal G} := \{k\ge0: C_k\le 0.9M_k\}, \quad {\cal B}:=\{k\ge0: C_k> 0.9M_k\},
\end{equation}
respectively.
The following result specializes the bound derived in Lemma \ref{lem:tech2} to the two exclusive cases
in which $k \in {\cal G}$ and $k \in {\cal B}$. More specifically, it derives a controllable
bound on the residual vector $v_{k+1}$ and the potential function
difference $\eta_{k+1}(u) - \eta_k(u)$ in the good iterations
and a controllable bound only on $\eta_{k+1}(u) - \eta_k(u)$ in the bad iterations.

\begin{lemma}\label{lem:basic}
	The following statements hold for every $ u\in\dom h $ and $ k\ge 0 $:
	\begin{itemize}
		\item[(a)] if $k \in {\cal G}$ then
		\begin{equation}\label{ineq:case1}
		\frac{A_{k+1}}{72.2M_k}\|v_{k+1}\|^2\le \eta_k(u)- \eta_{k+1}(u)+\frac12 \bar m a_k D^2;
		\end{equation}
		\item[(b)] if $k \in {\cal B}$ then
		\begin{equation}\label{ineq:case2}
		0\le \eta_k(u)- \eta_{k+1}(u)+\frac1{2} \bar m a_k  D^2 + \frac{1-\gamma}{2\gamma}D^2.
		\end{equation}
	\end{itemize}
\end{lemma}
\begin{proof}
	(a) Let $ k\in{\cal G}$ be given and note that  \eqref{def:G} and \eqref{eq:ty} imply that
	$0.9 M_k \ge C_k$ and
	$y_{k+1}=y_{k+1}^g$ where $y_{k+1}^g=y(\tx_k;M_k)$ is as in \eqref{eq:update}. Hence, using the inequality in \eqref{eq:optcond}, and the definitions of $C_k$ and $F_k$ in \eqref{eq:C} and \eqref{eq:Fk}, respectively, we conclude that $ \|v_{k+1}\|\le 1.9M_k \|y_{k+1}-\tx_k\| $ and $ F_k\le C_k \le 0.9M_k $. The latter two conclusions and Lemma \ref{lem:tech2} then immediately imply that
\eqref{ineq:case1} holds.

	(b) Let $ k\in{\cal B} $ be given and note  that  \eqref{eq:ty} and \eqref{def:G} imply that
	$y_{k+1}=y_{k+1}^b$. 
	Using the latter observation, Lemma \ref{lem:tech2}, Lemma \ref{lem:tech1}(c), the last equality in \eqref{eq:aktx}, and relation \eqref{eq:rel},
we conclude that
	\begin{align*}
      \eta_k(u)- \eta_{k+1}(u)+\frac{1}{2} \bar m a_k  D^2 &\ge
      \frac{(M_k-F_k)}{2}A_{k+1}\|y_{k+1}^b-\tx_k\|^2 \\
     &= \frac{(M_k-F_k)}{2}A_{k+1}\left \| \frac{A_ky_k+a_k x_{k+1}}{A_{k+1}} -  \frac{A_ky_k+a_k x_{k}}{A_{k+1}} \right \|^2 \\
&= \frac{(M_k-F_k) a_k^2}{2A_{k+1} }\|x_{k+1}-x_k\|^2 = \frac12 \left(1 - \frac{F_k}{M_k} \right)
\|x_{k+1}-x_k\|^2\\
&\ge \frac12 \left( 1 - \frac{1}{\gamma} \right) \|x_{k+1}-x_k\|^2,
	\end{align*}
     and hence that \eqref{ineq:case2} holds  in view of Lemma \ref{lem:tech1}(a) and (A3).
	\end{proof}

As a consequence, the next lemma provides the result of the summation of inequalities for $ k\in{\cal G} $ and $ k\in{\cal B} $ in Lemma \ref{lem:basic}.
\begin{lemma}\label{lem:sum}
	For every $ u\in\dom h $ and $ k\ge 1 $, we have
\begin{equation}\label{ineq:sum}
	\left( \frac{1}{36.1}\sum_{i\in {\cal G}_k}\frac{A_{i+1}}{M_i}\right)\min_{1\le i\le k}\|v_i\|^2
	 \le \|u-x_0\|^2-2\eta_k(u) + \bar mD^2A_k +\frac{1-\gamma}{\gamma}D^2|{\cal B}_k| ,
\end{equation}
where ${\cal G}_k$ and ${\cal B}_k$ are defined as
\begin{equation}
{\cal G}_k=\{i \in {\cal G} : i \le k-1 \}, \quad {\cal B}_k:=\{i \in {\cal B} : i\le k-1 \}.  \label{def:Gk}
\end{equation}
\end{lemma}
\begin{proof}
First, note that 
\[
\sum_{i\in {\cal G}_k}\frac{A_{i+1}}{M_i} \|v_{i+1}\|^2
\ge \left( \sum_{i\in {\cal G}_k}\frac{A_{i+1}}{M_i}\right)\min_{i\in {\cal G}_k}\|v_{i+1}\|^2
\ge \left( \sum_{i\in {\cal G}_k}\frac{A_{i+1}}{M_i}\right)\min_{1\le i\le k}\|v_i\|^2.
\]
The conclusion follows by  adding \eqref{ineq:case1} and \eqref{ineq:case2} both with $k=i$ as
$i$ varies in ${\cal G}_k$ and ${\cal B}_k$, respectively,
and using the above inequality, the definition of $\eta_k$ in \eqref{eq:eta}, and
the facts that $A_k=A_0+ \sum_{i=0}^{k-1}a_i$  and $A_0=0$, which are due to \eqref{eq:aktx} and
 step 0 of the AC-ACG method, respectively.
\end{proof}
Note that the left hand side of \eqref{ineq:sum} is actually  zero when ${\cal G}_k=\emptyset$, and hence 
\eqref{ineq:sum} is meaningless in this case.
The result below, which plays a major role in our analysis, uses for the first time the fact that $M_k$
is chosen as in \eqref{eq:M}  and shows that
${\cal G}_k$ is nonempty and well-populated.
This fact in turn implies that
the term inside the parenthesis in the left hand side of \eqref{ineq:sum} is sufficiently large
(see Lemma \ref{lem:est3} below). The proof of Theorem \ref{thm:main} will then follow by combining these observations.

\begin{lemma}\label{lem:card}
For every $k \ge 12$,  $|{\cal B}_k |\le k/3$
where ${\cal B}_k$ is as defined in \eqref{def:Gk}.
\end{lemma}

\begin{proof}
Let $k \ge 12$ be given and,
for the sake of this proof,
define $ C_{-1}^{avg}= 0 $.
In view of  \eqref{eq:M} and the definition of ${\cal B}_k$ in \eqref{def:Gk}, it follows that for every $i \in {\cal B}_k$,
\[
\frac{\alpha}{0.9} C_{i}> \alpha M_{i}\ge C_{i-1}^{avg}, 
\]
and hence that
\begin{align}\label{ineq:sum C}
	\frac{\alpha}{0.9} \sum_{i \in {\cal B}_k} C_i > 
\sum_{i \in {\cal B}_k} C^{avg}_{i-1}.
	\end{align}
Using Lemma \ref{lem:tech1}(c) and the facts that $ C_i>0.9M_i $ for every $ i\in{\cal B}_k $ and that $ M_i \ge \gamma M\ge \gamma \bar M $ (see \eqref{eq:M} and step 0 of the AC-ACG method) for every $ i\ge 0 $, we have 
\begin{equation}\label{ineq:lb}
0.9\gamma \bar M\le C_{i}\le \bar M \quad i\in{\cal B}_k .
\end{equation}
Let $l:=|{\cal B}_k|$ and let $i_1< \cdots < i_l$ denote the indices in ${\cal B}_k$.
	Clearly, in view of \eqref{eq:Cave} and the fact that $i_j \le k$ for every $j=1,\ldots,l$, we have
	\begin{align*}
	C^{avg}_{i_1-1}\ge 0, \quad
	C^{avg}_{i_2-1}\ge \frac1{k}  C_{i_1}, \quad \cdots\cdots, \quad
	C^{avg}_{i_l-1}\ge \frac{1}{k}\left( C_{i_1}+\cdots+C_{i_{l-1}} \right).
	\end{align*}
	Summing these inequalities, we obtain
	\begin{align*}
	 \sum_{i \in {\cal B}_k} C^{avg}_{i-1} \ge
	\frac{ 1}{k} \sum_{j=1}^{l}  (l-j) C_{i_j}\ge \frac{ 1}{k} \sum_{j=1}^{\lceil l/2 \rceil}  (l-j) C_{i_j}\ge \frac{1}{k}
	\left \lfloor \frac{l}2 \right \rfloor \sum_{j=1}^{\lceil l/2 \rceil} C_{i_j}.
	\end{align*}
	 Combining \eqref{ineq:sum C} and the last inequality, we then conclude that
	\[
	\frac{\alpha( S_1+ S_2)}{0.9} \ge \frac{1}{k} \left \lfloor \frac{l}2 \right \rfloor S_1
	\]
	where
	\begin{equation}\label{eq:S}
	S_1:= \sum_{j=1}^{\lceil l/2 \rceil} C_{i_j}, \quad S_2:=\sum_{j=\lceil l/2 \rceil+1}^{l} C_{i_j}.
	\end{equation}
 Since \eqref{ineq:lb} and the above definitions of $S_1$ and $S_2$ immediately imply
 that $S_2/S_1 \le 1/(0.9\gamma)$, we then conclude from the above inequality that
\begin{equation}\label{ineq:l1}
	|{\cal B}_k | =  l\le \left( \frac{2 \alpha k}{0.9} \right) \left( 1 + \frac{S_2}{S_1} \right) + 1 \le
	\left( \frac{2 \alpha k}{0.9} \right) \left( 1 + \frac1{0.9\gamma} \right) + 1
\end{equation}
and hence that  $|{\cal B}_k |\le k/4 +1\le k/3$ in view of the definition of $ \alpha $ in \eqref{eq:alpha} and the fact that $ k\ge 12 $.
The last conclusion of the lemma follows straightforwardly from the first one.
\end{proof}

We now make some remarks about choosing $\alpha$ more aggressively,
i.e., larger than the value in \eqref{eq:alpha}
(recall the discussion in the second paragraph following the AC-ACG method).
First, in view of their definitions in \eqref{eq:S},
the quantities $S_1$ and $S_2$ are actually quantities that
depend on the iteration index $k$ and hence should have been denoted by
$S_1^k$ and $S_2^k$. Second, it follows from the first inequality
in \eqref{ineq:l1} that
\[
|{\cal B}_k | \le \left( \frac{2 \alpha k}{0.9} \right)
\left( 1 + \bar \gamma_k \right) + 1 
\]
where $\bar \gamma_k := S_2^k/S_1^k$.
Third, we have used in the proof of Lemma \ref{lem:card} that  $\bar \gamma_k$
is bounded above by $1/(0.9 \gamma)$,
which is a very conservative bound for this quantity.
In practice though, $\bar \gamma_k$ behaves as ${\cal O}(1)$
(if not for all $k$, then at least for a substantial number of iterations). Fourth, in order to conclude
that $|B_k| \le k/3$ as in the proof of Lemma \ref{lem:card},
it suffices to choose
\[
\alpha = \frac{0.9}{8(1+\bar \gamma)}
\]
where $\bar \gamma := \max\{ \bar \gamma_k : k \ge 1\}$.
Observe that the above choice of $\alpha$ is
$\Theta(1)$ if $\bar \gamma$ behaves as ${\cal O}(1)$.

Before presenting Lemma \ref{lem:est3}, we first state two technical results about the
sequences $\{M_k\}$ and $\{A_k\}$.

\begin{lemma}\label{lem:avg}
	For every $1 \le i<k$, we have
	\[
	M_k \ge \frac{i}{k} M_i.
	\]
\end{lemma}
\begin{proof}
	From the definition of $ C_k^{avg} $ in \eqref{eq:Cave}, for every $i=1,\ldots,k-1$, we have
	\[
	k C_{k-1}^{avg} - i C_{i-1}^{avg} = C_{i}+ \ldots+C_{k-1}
	\]
	and thus
	\[
	\frac{C_{k-1}^{avg}}{C_{i-1}^{avg}} = \frac{i}{k} + \frac{C_{i}+ \ldots+C_{k-1}}{ k C_{i-1}^{avg}}\ge \frac ik.
	\]
	The conclusion follows from the above inequality, the definition of $ M_k $ in \eqref{eq:M} and the fact that $ \max\{a,c\}\ge\max\{b,d\} $ for $ a,b,c,d\in \R $ such that $ a\ge b $ and $ c\ge d $.
\end{proof}

The following result describes bounds on $A_k$ in terms of the
first $k$ elements of the sequence $\{M_i\}$ and also in terms of $M_k$ alone.

\begin{lemma}\label{estimates}
	Consider the sequences $ \{A_k\} $ and $ \{M_i\} $ defined in \eqref{eq:aktx} and \eqref{eq:M}, respectively.
	For every $ k\ge 12 $, we have
	\begin{equation}\label{ineq:est1}
	A_k
	\le \left( \sum_{i=0}^{k-1}\frac{1}{\sqrt{M_i}}\right)^2 \le  k \sum_{i=0}^{k-1}\frac{1}{M_i} \le
	k^2\frac{\theta_k}{M_k}
	\end{equation}
	and
	\begin{equation}\label{ineq:est2}
	A_k \ge \frac{1}{4}\left( \sum_{i=0}^{k-1}\frac{1}{\sqrt{M_i}}\right) ^2 \ge \frac{k^2}{12M_k}
	\end{equation}
	where $\theta_k$ is as in \eqref{def:theta}.
\end{lemma}
\begin{proof}
	We first establish the inequalities in \eqref{ineq:est1}.
	Using the first two identities in \eqref{eq:aktx} and the fact
	$ \sqrt{b_1 + b_2}\le\sqrt{b_1}+\sqrt{b_2} $ for any $ b_1,b_2 \in \R_+ $, we conclude that for any
	$ i \ge 0$,
	\begin{align*}
	\sqrt{ A_{i+1} }
	= \left( A_i+\frac{ 1 + \sqrt{1+4M_iA_i} }{2M_i} \right)^{1/2}
	\le \left( A_i+\frac{ 1 + \sqrt{M_iA_i} }{M_i} \right)^{1/2} \le \sqrt{A_i}+\frac{1}{\sqrt{M_i}}.
	\end{align*}
	Now, the first inequality in \eqref{ineq:est1} follows by
	summing the above inequality from $i=0$ to $k-1$ and using the assumption that $A_0=0$.
	Moreover, the second and third inequalities in \eqref{ineq:est1} follow straightforwardly from
	the Cauchy-Schwarz inequality and the definition of $\theta_k$ in \eqref{def:theta}, respectively.
	
	We now establish the inequalities in \eqref{ineq:est2}.
	Using the first two identities in \eqref{eq:aktx}, we have
	\begin{align*}
	\sqrt{ A_{i+1} }
	= \left( A_i+\frac{ 1 + \sqrt{1+4M_iA_i} }{2M_i} \right)^{1/2}
	\ge \left( A_i+\frac{ 1 + 2\sqrt{M_iA_i} }{2M_i} \right)^{1/2} \ge \sqrt{A_i}+\frac{1}{2\sqrt{M_i}}.
	\end{align*}
	The first inequality in \eqref{ineq:est2} now follows by
	summing the above inequality from $i=0$ to $k-1$ and using the assumption that $A_0=0$.
	For every $ k\ge 12 $, we have
	\[
	\sum_{i=1}^{k-1}\sqrt{i}\ge \int_{0}^{k-1}\sqrt{x}dx=\frac23(k-1)^{3/2} \ge\frac23\left( \frac {11}{12}k\right) ^{3/2}\ge0.58k^{3/2},
	\]
	which, together with Lemma \ref{lem:avg}, then implies that
	\[
	\sum_{i=1}^{k-1}\frac{1}{\sqrt{M_i}} \ge  \frac{1}{\sqrt{k M_k }} \sum_{i=1}^{k-1}\sqrt{i}
	\ge  \frac{0.58k}{\sqrt{M_k}}.
	\]
	The second inequality in \eqref{ineq:est2} now follows immediately from the one above.
\end{proof}

The following result provides a lower bound on
the term inside the parentheses of the left hand side of \eqref{ineq:sum}.

\begin{lemma}\label{lem:est3}
	For every $ k\ge 12 $, we have
	\[
	\sum_{i\in {\cal G}_k}\frac{A_{i+1}}{M_i} \ge  \frac{k^3}{3402 M_k^2}.
	\]
\end{lemma}
\begin{proof}
	Let $ k\ge12 $ be given and define
	\begin{equation}\label{def:tGk}
	\tilde {\cal G}_k:=\{i\in{\cal G}_k :  i\ge \lfloor k/3 \rfloor \}, \quad \tilde {\cal B}_k:=\{i\in{\cal B}_k: i\ge \lfloor k/3 \rfloor \}.
	\end{equation}
	Using Lemma \ref{lem:avg}, the facts that $\tilde {\cal G}_k\subset {\cal G}_k$,
	$ \{A_k\} $ is strictly increasing,
	and $i/k \ge 2/7 $ for any $ i\in\tilde {\cal G}_k $ and $ k\ge12 $, and inequality \eqref{ineq:est2},
	we conclude that
	\begin{align}
	\sum_{i\in {\cal G}_k}\frac{A_{i+1}}{M_i} 
	&\ge  \sum_{i\in {\cal G}_k} \frac{iA_{i+1}}{kM_k}
	\ge  \sum_{i\in \tilde {\cal G}_k} \frac{iA_{i+1}}{kM_k}
	\ge \frac{2|\tilde {\cal G}_k|}{7 M_k} A_{\lfloor k/3\rfloor+1} \nn \\
	& \ge
	\frac{2|\tilde {\cal G}_k|}{7 M_k} A_{\lceil k/3\rceil} \ge
	\frac{|\tilde {\cal G}_k| ( \lceil k/3\rceil)^2}{42 M_k M_{\lceil k/3\rceil}} \ge
	\frac{|\tilde {\cal G}_k| k^2}{378 M_k M_{\lceil k/3\rceil}} \label{ineq:key}.
	\end{align}
	On the other hand,  Lemma \ref{lem:avg} with $i = \lceil k/3 \rceil$ implies that
	\[
	M_k \ge \frac{\lceil k/3\rceil}{k} M_{\lceil k/3\rceil} \ge \frac13 M_{\lceil k/3\rceil}.
	\]
	Moreover, the definition of $\tilde {\cal G}_k$ in \eqref{def:tGk}, the fact that $\tilde {\cal B}_k \subset {\cal B}_k$ and Lemma \ref{lem:card} imply that 
	\[
	|\tilde {\cal G}_k|=k-\lfloor k/3\rfloor-|\tilde {\cal B}_k| \ge k-\lfloor k/3\rfloor-| {\cal B}_k|
	\ge k/3.
	\]
	The conclusion of the lemma now follows by combining \eqref{ineq:key} with the last two
	observations.
\end{proof}

We are now ready to prove the main result  of our paper.

\vgap

\noindent
{\bf Proof of Theorem \ref{thm:main}:}
	(a) The conclusion immediately follows from Lemma \ref{lem:tech1}(d).

	(b) Letting $x_*\in X_*$ be given and noting that $\eta_k(x_*) \ge 0$ in view of
	the definition of $\eta_k$ in \eqref{eq:eta} and using the above inequality,
	Lemma \ref{lem:sum} with $ u=x_* $, Lemma \ref{lem:card} and relation \eqref{ineq:est1},
	we conclude that
	\begin{align*}
	\left( \frac{1}{36.1}\sum_{i\in {\cal G}_k}\frac{A_{i+1}}{M_i} \right) \min_{1\le i\le k}\|v_i\|^2
	&\le \|x_0-x_*\|^2+ \bar 
	mD^2A_k+\frac{1-\gamma}{\gamma}D^2|{\cal B}_k|\\
	&\le D^2 + \bar mD^2 A_k+\frac{ (1-\gamma)D^2k}{3\gamma} \\
	&\le D^2+ \frac{\bar mD^2k^2\theta_k}{M_k}+\frac{ (1-\gamma)D^2k}{3\gamma}.
	\end{align*}
Statement b) of the theorem now follows by combining the above inequality and Lemma \ref{lem:est3}.
\QEDA

\section{Numerical results}\label{sec:numerics}
This section presents computational results to illustrate the performance of two variants
of the AC-ACG method against
five other state-of-the-art algorithms on a collection of nonconvex optimization problems that are either in the form of or can be easily reformulated
into \eqref{eq:PenProb2Intro}.
It contains five subsections, with each one
reporting computational results on one of
following classes of nonconvex optimization problems:
(a) quadratic programming (Subsection \ref{subsec:QP}); (b) support vector machine (SVM, Subsection \ref{subsec:SVM}); (c) sparse PCA (Subsection \ref{subsec:SPCA}); (d) matrix completion (Subsection \ref{subsec:MC}); and (e)
nonnegative matrix factorization (NMF, Subsection \ref{subsec:NMF}). 
Note that sparse PCA and NMF are problems
for which $\dom h$ is unbounded.

We start by describing the two AC-ACG variants considered in our computational experiments, both of which do not impose the restrictive condition \eqref{eq:alpha} on the choice
of $\alpha$ and $\gamma$.
The first variant, which we refer to as ACT throughout this section, preserves all steps in the AC-ACG method except that
$\gamma$ and $\alpha$ are provided as input by the user
without necessarily satisfying \eqref{eq:alpha}.
In our implementation, we set $\gamma=0.01$ for every
problem class listed above but the one in (b)
for which $\gamma$ is set to $0.002$. The latter choice
of $\gamma$ prevents the percentage of good iterations from being 100\% all the time and instead keeps it
within a range of about 65\% to 75\% (see Subsection
\ref{subsec:SVM}). The choice of the scalar $\alpha$
varies per problem class and is described in each one of
the subsections below.
The second variant, referred to as AC throughout this section,
sets $M_0=0.01M$, and computes
$M_{k+1}$ as in \eqref{eq:M} with $\gamma=10^{-6}$ and $C_k$ as
\begin{equation}\label{eq:AC}
C_k= \max\{ {\cal C}(y_{k+1}^g;\tx_k), 0\} 
\end{equation}
where $ {\cal C}(\cdot;\cdot) $
is defined in \eqref{eq:cal C}.
Our implementation of AC sets $ \alpha $ to values that
depend on the problem class under consideration and are specified in the subsections below.
Clearly, among the two variants described above,
ACT is the closest to AC-ACG.

We compare the two variants of AC-ACG with five other methods, namely: 
(i) the AG method proposed in \cite{nonconv_lan16};
(ii) the NC-FISTA of \cite{liang2019fista};
(iii) the ADAP-NC-FISTA also described in \cite{liang2019fista};
(iv) the NM-APG method proposed in \cite{Li_Lin2015};
and (v) the UPFAG method in \cite{LanUniformly}.
We remark that methods (i)-(iii) are the three pure ACG variants that have been outlined in Subsection \ref{subsec:pure} and methods (iv) and (v) are two
among the three hybrid ACG variants that have been discussed in Subsection \ref{subsec:hybrid}.
For the sake of simplicity, we use the
abbreviations  NM, UP, NC and AD to refer to
the NM-APG, UPFAG, NC-FISTA and ADAP-NC-FISTA methods, respectively, both in the discussions and tables below.



This paragraph provides details about
the three pure ACG variants used in our benchmark.
AG was implemented by the authors based on its description provided in Algorithm 1 of \cite{nonconv_lan16}
where the sequences
$ \{\alpha_k \} $, $ \{\beta_k \} $ and $ \{\lam_k \} $ were chosen as
$ \alpha_k = 2/(k+1) $, $ \beta_k = 0.99/M $ and $ \lam_k = k\beta_k/2 $, respectively,
and the Lipschitz constant $ M $ was computed as described in each of the five subsections below.
We note that the choice $\beta_k=0.99/M$ used in our implementation differs from
the one suggested in \cite{nonconv_lan16}, namely, $\beta_k=0.5/M$ (see (2.27) of \cite{nonconv_lan16}),
and consistently improves the practical performance of AG.
The NC and AD variants were also implemented by
the authors based on their descriptions in \cite{liang2019fista}. The triple $ (M, m, A_0) $ needed as input by NC was set to $ (M/0.99, m, 5000) $ where $m \ge \bar m$ (see the first inequality in \eqref{ineq:curv}).
The triple $ (M_0, m_0, \theta) $ needed as input by AD was 
set to $ (1,1000,1.25) $ in Subsections \ref{subsec:QP}, \ref{subsec:SVM} and \ref{subsec:NMF}, and $ (1,1,1.25) $ in Subsections \ref{subsec:SPCA} and \ref{subsec:MC}.

This paragraph provides implementation details for the two hybrid ACG variants used in our benchmark.
The NM method was implemented by the authors based on its description provided in
Algorithm 2 of \cite{Li_Lin2015}
which does not use line searches to compute
$M_k^a$ and $M_k^{na}$.
More specifically,
the quadruple $ (\alpha_x,\alpha_y,\eta,\delta) $ needed
as input by Algorithm 2 of \cite{Li_Lin2015}
was set to $ (0.99/M,0.99/M,0.8,1) $.
The code for UP was made available by the authors of \cite{LanUniformly} where
UP is described (see Algorithm 1 of \cite{LanUniformly}). In particular, we have used their choice of parameters
but have modified the code slightly to accommodate for the termination criterion \eqref{incl:v} used in our benchmark.
More specifically, the parameters $ (\hat \lam_0, \hat \beta_0, \gamma_1,\gamma_2,\gamma_3,\delta,\sigma)$  needed as input by UP
were set to
$ (1/M,1/M,1,1,1,10^{-3},10^{-10}) $.
Recall that UP computes the good upper curvatures
$M_k^a$ and $M_k^{na}$ by line searches (see Subsection \ref{subsec:hybrid}).
Our implementation of UP initiates
these scalars in both line searches by
using a Barzilai-Borwein type strategy
(see equation (2.12) in \cite{LanUniformly}).



All seven methods terminate with a pair $(z,v)$ satisfying
\[
v\in \nabla f(z)+\partial h(z), \qquad  \frac{\|v\|}{\|\nabla f(z_{0})\|+1}\leq \hat \rho,
\]
where $ \hat \rho=5\times 10^{-4} $ in the matrix completion problem and $ \hat \rho=10^{-7} $ in all the other problems.
All the computational results were obtained using
MATLAB R2017b on a MacBook Pro with a quad-core Intel Core i7 processor and 16 GB of memory.

\subsection{Quadratic programming} \label{subsec:QP}


This subsection discusses the performance of the AC-ACG method for solving a class of quadratic programming problems.

More specifically, it considers the problem
\begin{equation}\label{testQPprobMat}
\min\left\{ f(Z):=-\frac{\alpha_1}{2}\|D\mathcal{B}(Z)\|^{2}+\frac{\alpha_2}{2}\|\mathcal{A}(Z)-b\|^{2}:Z\in P_{n}\right\}
\end{equation}
where $(\alpha_1,\alpha_2)\in \R^2_{++}$, $b\in \R^{ l}$ is a vector with entries sampled
from the uniform distribution ${\cal U}[0,1]$, 
$D\in \R^{n\times n}$ is a diagonal matrix whose diagonal entries
are sampled from the discrete uniform distribution ${\cal U}\{1,1000\}$,
$ P_n:=\{Z\in {\cal S}_+^n:\text{tr}(Z)=1\} $ denotes the spectraplex, and 
$ \mathcal{A}:{\cal S}_+^n\rightarrow \R^l $ and $ \mathcal{B}:{\cal S}_+^n\rightarrow \R^n $ are linear operators given by
\begin{align*}
\left[ \mathcal{A}(Z) \right]_i&=
\inner{A_i}{Z}_F \quad \forall \, 1\le i\le l, \\
\left[ \mathcal{B}(Z) \right]_j&=\inner{B_j}{Z}_F \quad \forall \, 1\le j\le n, 
\end{align*}
with $ A_i \in {\cal S}^n_+$ and
$ B_j \in {\cal S}^n_+$ all being sparse matrices having the same density (i.e., percentage of
nonzeros) $ d $ and nonzero entries
uniformly sampled from $[0,1]$.

The quadratic programming problem \eqref{testQPprobMat} is an instance of \eqref{eq:PenProb2Intro} where $ h$ is the indicator function of the spectraplex $ P_n $.
For chosen curvature pairs $(M,m)\in\R^2_{++}$, the scalars $\alpha_1$ and $\alpha_2$ are chosen so that
$\lambda_{\max}(\nabla^{2}f)=M$ and $\lambda_{\min}(\nabla^{2}f)=-m$ where $\lambda_{\max}(\cdot)$ and $\lambda_{\min}(\cdot)$
denote the largest and smallest eigenvalue functions, respectively.

We start all seven methods from the same initial point $ Z_0 = I_n/n $ where $ I_n $ is an $ n\times n $ identity matrix, namely $ Z_0 $ is the centroid of $ P_n $.
The parameter $ \alpha$ is set to $1 $ in AC and
0.5 in ACT.

Numerical results for the seven methods are given
in Tables \ref{tab:t1}, \ref{tab:t2} and \ref{tab:t3},
with each table addressing a collection of
instances with the same dimension pair $(l,n)$ and
density $ d $.
Specifically, each row of Tables \ref{tab:t1}, \ref{tab:t2} and \ref{tab:t3} corresponds to an
instance of \eqref{testQPprobMat}, their first column
specifies the pair $(M,m)$ for
the corresponding instance,
their second to sixth
(resp., seventh to eleventh) columns
provide numbers of iterations (resp., running times) for the seven methods.
The best objective function values obtained
by all seven methods are not reported
since they are essentially the same
on all instances.
The number of resolvent evaluations is 1 in NC, 2 in AG, AC and ACT, 1 or 2 in NM, 1 on average in AD, and 3 on average in UP.
The bold numbers highlight the method that has the best performance in an instance of the problem.

Some statistic measures for AC and ACT to solve
the instances in Tables \ref{tab:t1}, \ref{tab:t2} and \ref{tab:t3} are given in Tables \ref{tab:t1-1}, \ref{tab:t2-1} and \ref{tab:t3-1}, respectively.
The first column in these tables
is the same as that of Tables \ref{tab:t1}, \ref{tab:t2} and \ref{tab:t3}, 
the second (resp., fifth) column provides the maximum of all observed curvatures $  C_k $ in AC (resp. ACT),
the third (resp., sixth) column provides the average of all observed curvatures $ C_k $ in AC (resp. ACT),
and the fourth (resp. seventh) column gives the percentage of good iterations (see \eqref{def:G}) in AC (resp. ACT).

%

%
%

\vspace{2mm}
In Tables \ref{tab:t1}-\ref{tab:t1-1}, the density $ d=2.5\% $ and the dimension pair $ (l, n) = (50, 200)$.
\vspace{-2mm}
\begin{table}[H]
	\begin{centering}
		\begin{tabular}{|>{\centering}p{1.5cm}|>{\centering}p{.6cm}>{\centering}p{.8cm}>{\centering}p{.8cm}>{\centering}p{1.8cm}>{\centering}p{1.8cm}|>{\centering}p{.6cm}>{\centering}p{.6cm}>{\centering}p{.6cm}>{\centering}p{1.5cm}>{\centering}p{1.5cm}|}
			\hline 
			{$(M,m)$}
			& \multicolumn{5}{c|}{Iteration Count} & \multicolumn{5}{c|}{Running Time (s)}
			\tabularnewline  
			\cline{2-11}
			&  {\small{}AG} & {\small{}NM} & {\small{}UP}& {\small{}NC/AD} & {\small{}ACT/AC} & {\small{}AG} & {\small{}NM} & {\small{}UP}& {\small{}NC/AD} & {\small{}ACT/AC}
			\tabularnewline
			\hline 
			{\small{}$(10^6,10^6)$}  & {\small{}46} & {\small{}80} & {\small{}9} & {\small{}33/12} & {\small{}23/8} & {\small{}1.6} & {\small{}2.1} & {\small{}0.7} &{\small{}}0.8/0.7& {\small{}1.4/\textbf{0.6}}  \tabularnewline
			\hline
			{\small{}$(10^6,10^5)$} & {\small{}3089} & {\small{}6242} & {\small{}2633} & {\small{}3384/2206} & {\small{}1009/883} & {\small{}130} & {\small{}191} &{\small{}261} &{\small{}94/89} &{\small{}57/\textbf{39}}  \tabularnewline
			\hline
			{\small{}$(10^6,10^4)$}  & {\small{}5400} & {\small{}10404} & {\small{}7203} & {\small{}1236/2591} & {\small{}1820/1760} & {\small{}188} & {\small{}328} &{\small{}705} &{\small{}\textbf{30}/104} &{\small{}109/{73}}  \tabularnewline
			\hline
			{\small{}$(10^6,10^3)$}  & {\small{}4621} & {\small{}11053} & {\small{}5429} & {\small{}5139/2637} & {\small{}1712/1508} & {\small{}176} & {\small{}360} &{\small{}540} &{\small{}122/109} & {\small{}118/\textbf{68}} \tabularnewline
			\hline
			{\small{}$(10^6,10^2)$} & {\small{}4476} & {\small{}11271} & {\small{}6891} & {\small{}11838/2639} & {\small{}1610/1472} & {\small{}176} & {\small{}312} &{\small{}653} &{\small{}283/116} & {\small{}103/\textbf{65}} \tabularnewline
			\hline
			{\small{}$(10^6,10)$} & {\small{}4461} & {\small{}11253} & {\small{}6479} & {\small{}14851/2640} & {\small{}1599/1485} & {\small{}171} & {\small{}311} &{\small{}613} &{\small{}362/116} & {\small{}155/\textbf{66}} \tabularnewline
			\hline
		\end{tabular}
		\par\end{centering}
	\caption{Numerical results for AG, NM, UP, NC, AD, ACT and AC}\label{tab:t1}
\end{table}
\vspace{-5mm}

\begin{table}[H]
	\begin{centering}
		\begin{tabular}{|>{\centering}p{1.5cm}|>{\centering}p{1cm}>{\centering}p{1cm}>{\centering}p{1cm}|>{\centering}p{1cm}>{\centering}p{1cm}>{\centering}p{1cm}|}
			\hline 
			{$(M,m)$}
			& \multicolumn{3}{c|}{AC} & \multicolumn{3}{c|}{ACT}
			\tabularnewline  
			\cline{2-7}
			& {\small{}Max} & {\small{}Avg} & {\small{}Good} & {\small{}Max} & {\small{}Avg} & {\small{}Good}
			\tabularnewline
			\hline 
			{\small{}$(10^6,10^6)$} & {\small{}1.88E5} & {\small{}3.04E4} & {\small{}88\%} & {\small{}8.38E5} & {\small{}1.53E5} & {\small{}95\%}
			\tabularnewline \hline
			{\small{}$(10^6,10^5)$} & {\small{}4.85E5} & {\small{}8.84E4} & {\small{}64\%} & {\small{}7.00E5} & {\small{}9.25E4} & {\small{}98\%}
			\tabularnewline
			\hline
			{\small{}$(10^6,10^4)$} & {\small{}5.42E5} & {\small{}1.24E5} & {\small{}65\%} & {\small{}7.24E5} & {\small{}1.04E5} & {\small{}99\%}
			\tabularnewline \hline
			{\small{}$(10^6,10^3)$} & {\small{}5.48E5} & {\small{}1.20E5} & {\small{}69\%} & {\small{}7.27E5} & {\small{}1.16E5} & {\small{}97\%}
			\tabularnewline
			\hline
			{\small{}$(10^6,10^2)$} & {\small{}5.49E5} & {\small{}1.20E5} & {\small{}68\%} & {\small{}7.27E5} & {\small{}1.10E5} & {\small{}99\%}
			\tabularnewline
			\hline
			{\small{}$(10^6,10)$} & {\small{}5.49E5} & {\small{}1.18E5} & {\small{}70\%} & {\small{}7.27E5} & {\small{}1.09E5} & {\small{}99\%}
			\tabularnewline
			\hline
		\end{tabular}
		\par\end{centering}
	\caption{AC and ACT statistics}\label{tab:t1-1}
\end{table}

In Tables \ref{tab:t2}-\ref{tab:t2-1}, the density $ d=0.5\% $ and the dimension pair $ (l, n) = (50, 400)$.
\vspace{-2mm}
\begin{table}[H]
	\begin{centering}
		\begin{tabular}{|>{\centering}p{1.5cm}|>{\centering}p{.6cm}>{\centering}p{.8cm}>{\centering}p{.8cm}>{\centering}p{1.8cm}>{\centering}p{1.8cm}|>{\centering}p{.6cm}>{\centering}p{.6cm}>{\centering}p{.6cm}>{\centering}p{1.5cm}>{\centering}p{1.5cm}|}
			\hline 
			{$(M,m)$}
			& \multicolumn{5}{c|}{Iteration Count} & \multicolumn{5}{c|}{Running Time (s)}
			\tabularnewline  
			\cline{2-11}
			&  {\small{}AG} & {\small{}NM} & {\small{}UP}& {\small{}NC/AD} & {\small{}ACT/AC} & {\small{}AG} & {\small{}NM} & {\small{}UP}& {\small{}NC/AD} & {\small{}ACT/AC}
			\tabularnewline
			\hline 
			{\small{}$(10^6,10^6)$}  & {\small{}44} & {\small{}75} & {\small{}10} & {\small{}33/12} & {\small{}17/8} & {\small{}4.4} & {\small{}5.1} & {\small{}1.9} &{\small{}}2.1/1.8& {\small{}2.6/\textbf{1.0}}  \tabularnewline
			\hline
			{\small{}$(10^6,10^5)$} & {\small{}1411} & {\small{}3151} & {\small{}56} & {\small{}610/530} & {\small{}403/131} & {\small{}134} & {\small{}224} & \textbf{\small{}13} &{\small{}39/56} &{\small{}58/\small{}16}  \tabularnewline
			\hline
			{\small{}$(10^6,10^4)$}  & {\small{}1963} & {\small{}5071} & {\small{}105} & {\small{}1212/868} & {\small{}599/237} & {\small{}195} & {\small{}373} & \textbf{\small{}26} &{\small{}76/93} &{\small{}88/\small{}28}  \tabularnewline
			\hline
			{\small{}$(10^6,10^3)$}  & {\small{}1935} & {\small{}5172} & {\small{}115} & {\small{}4415/900} & {\small{}564/245} & {\small{}193} & {\small{}382} & \textbf{\small{}29} &{\small{}277/103} & {\small{}95/30} \tabularnewline
			\hline
			{\small{}$(10^6,10^2)$} & {\small{}1934} & {\small{}5045} & {\small{}119} & {\small{}7325/904} & {\small{}559/242} & {\small{}190} & {\small{}367} & {\small{}32} &{\small{}465/103} & {\small{}91/\textbf{29}} \tabularnewline
			\hline
			{\small{}$(10^6,10)$} & {\small{}1934} & {\small{}5056} & {\small{}113} & {\small{}7527/904} & {\small{}561/246} & {\small{}194} & {\small{}373} & {\small{}31} &{\small{}477/104} & {\small{}92/\textbf{29}} \tabularnewline
			\hline
		\end{tabular}
		\par\end{centering}
	\caption{Numerical results for AG, NM, UP, NC, AD, ACT and AC}\label{tab:t2}
\end{table}
\vspace{-5mm}

\begin{table}[H]
	\begin{centering}
		\begin{tabular}{|>{\centering}p{1.5cm}|>{\centering}p{1cm}>{\centering}p{1cm}>{\centering}p{1cm}|>{\centering}p{1cm}>{\centering}p{1cm}>{\centering}p{1cm}|}
			\hline 
			{$(M,m)$}
			& \multicolumn{3}{c|}{AC} & \multicolumn{3}{c|}{ACT}
			\tabularnewline  
			\cline{2-7}
			& {\small{}Max} & {\small{}Avg} & {\small{}Good} & {\small{}Max} & {\small{}Avg} & {\small{}Good}
			\tabularnewline
			\hline 
			{\small{}$(10^6,10^6)$} & {\small{}2.40E5} & {\small{}3.22E4} & {\small{}88\%} & {\small{}6.32E5} & {\small{}1.67E5} & {\small{}93\%}
			\tabularnewline \hline
			{\small{}$(10^6,10^5)$} & {\small{}1.53E5} & {\small{}1.98E4} & {\small{}76\%} & {\small{}4.05E5} & {\small{}5.92E4} & {\small{}99\%}
			\tabularnewline
			\hline
			{\small{}$(10^6,10^4)$} & {\small{}2.03E5} & {\small{}2.50E4} & {\small{}72\%} & {\small{}4.16E5} & {\small{}6.66E4} & {\small{}98\%}
			\tabularnewline \hline
			{\small{}$(10^6,10^3)$} & {\small{}2.07E5} & {\small{}2.55E4} & {\small{}69\%} & {\small{}4.17E5} & {\small{}6.58E4} & {\small{}98\%}
			\tabularnewline
			\hline
			{\small{}$(10^6,10^2)$} & {\small{}2.08E5} & {\small{}2.55E4} & {\small{}71\%} & {\small{}4.17E5} & {\small{}6.54E4} & {\small{}98\%}
			\tabularnewline
			\hline
			{\small{}$(10^6,10)$} & {\small{}2.08E5} & {\small{}2.51E4} & {\small{}68\%} & {\small{}4.17E5} & {\small{}6.56E4} & {\small{}98\%}
			\tabularnewline
			\hline
		\end{tabular}
		\par\end{centering}
	\caption{AC and ACT statistics}\label{tab:t2-1}
\end{table}

In Tables \ref{tab:t3}-\ref{tab:t3-1}, the density $ d=0.1\% $ and the dimension pair $ (l, n) = (50, 800)$.
\vspace{-2mm}
\begin{table}[H]
	\begin{centering}
		\begin{tabular}{|>{\centering}p{1.5cm}|>{\centering}p{.6cm}>{\centering}p{.8cm}>{\centering}p{.8cm}>{\centering}p{1.8cm}>{\centering}p{1.8cm}|>{\centering}p{.6cm}>{\centering}p{.6cm}>{\centering}p{.6cm}>{\centering}p{1.5cm}>{\centering}p{1.5cm}|}
			\hline 
			{$(M,m)$}
			& \multicolumn{5}{c|}{Iteration Count} & \multicolumn{5}{c|}{Running Time (s)}
			\tabularnewline  
			\cline{2-11}
			&  {\small{}AG} & {\small{}NM} & {\small{}UP}& {\small{}NC/AD} & {\small{}ACT/AC} & {\small{}AG} & {\small{}NM} & {\small{}UP}& {\small{}NC/AD} & {\small{}ACT/AC}
			\tabularnewline
			\hline 
			{\small{}$(10^6,10^6)$}  & {\small{}69} & {\small{}117} & {\small{}13} & {\small{}38/11} & {\small{}18/8} & {\small{22}} & {\small{}26} & {\small{}8} &{\small{}8/7}& {\small{}13/\textbf{4}}  \tabularnewline
			\hline
			{\small{}$(10^6,10^5)$} & {\small{}277} & {\small{}502} & {\small{}9} & {\small{}176/24} & {\small{}31/7} & {\small{119}} & {\small{}118} & {\small{}6} &{\small{}36/10} &{\small{}20/\textbf{3}}  \tabularnewline
			\hline
			{\small{}$(10^6,10^4)$}  & {\small{}491} & {\small{}1030} & {\small{}13} & {\small{}786/60} & {\small{}65/11} & {\small{173}} & {\small{}246} & {\small{}9} &{\small{}163/24} &{\small{}39/\textbf{5}}  \tabularnewline
			\hline
			{\small{}$(10^6,10^3)$}  & {\small{}531} & {\small{}1144} & {\small{}13} & {\small{}1519/70} & {\small{}67/12} & {\small{169}} & {\small{}259} & {\small{}9} &{\small{}313/28} & {\small{}41/\textbf{7}} \tabularnewline
			\hline
			{\small{}$(10^6,10^2)$} & {\small{}535} & {\small{}1156} & {\small{}13} & {\small{}1698/71} & {\small{}67/12} & {\small{172}} & {\small{}260} & {\small{}9} &{\small{}351/28} & {\small{}43/\textbf{6}} \tabularnewline
			\hline
			{\small{}$(10^6,10)$} & {\small{}536} & {\small{}1157} & {\small{}13} & {\small{}1703/71} & {\small{}67/12} & {\small{172}} & {\small{}266} & {\small{}8} &{\small{}352/28} & {\small{}44/\textbf{5}} \tabularnewline
			\hline
		\end{tabular}
		\par\end{centering}
	\caption{Numerical results for AG, NM, UP, NC, AD, ACT and AC}\label{tab:t3}
\end{table}
\vspace{-5mm}

\begin{table}[H]
	\begin{centering}
		\begin{tabular}{|>{\centering}p{1.5cm}|>{\centering}p{1cm}>{\centering}p{1cm}>{\centering}p{1cm}|>{\centering}p{1cm}>{\centering}p{1cm}>{\centering}p{1cm}|}
			\hline 
			{$(M,m)$}
			& \multicolumn{3}{c|}{AC} & \multicolumn{3}{c|}{ACT}
			\tabularnewline  
			\cline{2-7}
			& {\small{}Max} & {\small{}Avg} & {\small{}Good} & {\small{}Max} & {\small{}Avg} & {\small{}Good}
			\tabularnewline
			\hline 
			{\small{}$(10^6,10^6)$} & {\small{}1.28E5} & {\small{}1.70E4} & {\small{}88\%} & {\small{}3.65E5} & {\small{}5.37E4} & {\small{}94\%}
			\tabularnewline \hline
			{\small{}$(10^6,10^5)$} & {\small{}1.80E4} & {\small{}2.84E3} & {\small{}86\%} & {\small{}1.78E5} & {\small{}2.64E4} & {\small{}96\%}
			\tabularnewline
			\hline
			{\small{}$(10^6,10^4)$} & {\small{}3.26E4} & {\small{}3.89E3} & {\small{}91\%} & {\small{}1.78E5} & {\small{}2.99E4} & {\small{}98\%}
			\tabularnewline \hline
			{\small{}$(10^6,10^3)$} & {\small{}3.41E4} & {\small{}3.73E3} & {\small{}92\%} & {\small{}1.78E5} & {\small{}2.62E4} & {\small{}98\%}
			\tabularnewline
			\hline
			{\small{}$(10^6,10^2)$} & {\small{}3.42E4} & {\small{}3.75E3} & {\small{}92\%} & {\small{}1.78E5} & {\small{}2.58E4} & {\small{}98\%}
			\tabularnewline
			\hline
			{\small{}$(10^6,10)$} & {\small{}3.43E4} & {\small{}3.75E3} & {\small{}92\%} & {\small{}1.78E5} & {\small{}2.57E4} & {\small{}98\%}
			\tabularnewline
			\hline
		\end{tabular}
		\par\end{centering}
	\caption{AC and ACT statistics}\label{tab:t3-1}
\end{table}
\vspace{-3mm}

In summary, computational results demonstrate that: i) the computed
average curvature of AC is small compared with $ M $ and the computed maximum curvature; 
ii) the percentage of good iterations of AC lies in a suitable
range; and iii) AC has the best performance in terms of running time.

\subsection{Support Vector Machine} \label{subsec:SVM}
This subsection presents the performance of AC-ACG for solving a support vector machine problem. Given data points $ \{(x_i,y_i)\}_{i=1}^p $, where $ x_i\in \R^n $ is a feature vector and $ y_i\in \{-1,1\} $ denotes the corresponding label, we consider the SVM
problem defined as
\begin{equation}\label{eq:SVM}
\min_{z\in \R^n} \frac1p \sum_{i=1}^{p} \ell(x_i,y_i;z)+\frac{\lam}{2}\|z\|^2+I_{{\cal B}_r} (z)
\end{equation}
for some $ \lam, r>0 $, where $ \ell(x_i,y_i;\cdot)=1-\tanh(y_i\inner{\cdot}{x_i}) $ is a nonconvex sigmoid loss function and $ I_{{\cal B}_r} (\cdot) $ is the indicator function of the ball $ B_r:=\{z\in \R^n:\|z\|\le r\} $.
The SVM problem \eqref{eq:SVM} is an instance of nonconvex SCO problems \eqref{eq:PenProb2Intro} where
\[
f(z)=\frac1p \sum_{i=1}^{p} \ell(x_i,y_i;z)+\frac{\lam}{2}\|z\|^2, \quad h(z)=I_{{\cal B}_r} (z).
\]
Clearly, $ f $ is differentiable everywhere and its gradient is $M$-Lipschitz continuous where
\begin{equation}\label{eq:M-SVM}
M = \frac1p \sum_{i=1}^p L_i + \lam, \quad L_i=\frac{4\sqrt{3}}{9}\|x_i\|^2 \ \ \forall i=1,\ldots,p.
\end{equation}
Since no sharper $m<M$ satisfying the first inequality in \eqref{ineq:curv}
is known, we simply set $m=M$.

We generate synthetic data sets as follows: for each data point $ (x_i , y_i) $, $ x_i $ is drawn from the uniform distribution on $ [0,1]^n$ and is sparse with 5\% nonzero components, and $ y_i=\text{sign}(\inner{\bar z}{x_i}) $ for some $ \bar z \in B_r$. We consider four different problem sizes $ (n,p)$, i.e., $ (1000,500) $, $ (2000, 1000) $, $ (3000, 1000) $ and $ (4000, 500 )$. We set $ \lam = 1/p$ and $ r = 50 $.

We start all seven methods from the same initial point $ z_0 $ that
is chosen randomly from the uniform distribution within the ball $ B_r $. 
The parameter $ \alpha$ is set to $ 0.5 $ in both AC and ACT.

Numerical results of the seven methods are given in Table \ref{tab:t4} and some statistic measures of AC and ACT are given in Table \ref{tab:t4-1}. The explanation of
their columns excluding the first one is the same as those of Tables \ref{tab:t1}-\ref{tab:t3-1} (see the two paragraphs preceding Table \ref{tab:t1}).
Their first columns differ from those of Tables \ref{tab:t1}-\ref{tab:t3-1} in that they only list the value of $M$ computed according to \eqref{eq:M-SVM}.
The best objective function values obtained by all seven methods are not reported since they are essentially the same on all instances.
The number of resolvent evaluations is 1 in NC, 2 in AG, AC and ACT, 1 or 2 in NM, 1 on average in AD, and 3 on average in UP.
The bold numbers highlight the method that has the best performance in an instance of the problem. The numbers marked with * indicate that the maximum number of iterations has been reached.




\begin{table}[H]
	\begin{centering}
		\begin{tabular}{|>{\centering}p{.5cm}|>{\centering}p{1cm}>{\centering}p{1cm}>{\centering}p{1cm}>{\centering}p{2cm}>{\centering}p{2cm}|>{\centering}p{.6cm}>{\centering}p{.6cm}>{\centering}p{.6cm}>{\centering}p{1.5cm}>{\centering}p{1.5cm}|}
			\hline 
			{$M$}
			& \multicolumn{5}{c|}{Iteration Count} & \multicolumn{5}{c|}{Running Time (s)}
			\tabularnewline  
			\cline{2-11}
			&  {\small{}AG} & {\small{}NM} & {\small{}UP}& {\small{}NC/AD} & {\small{}ACT/AC} & {\small{}AG} & {\small{}NM} & {\small{}UP}& {\small{}NC/AD} & {\small{}ACT/AC}
			\tabularnewline
			\hline 
			{\small{}$13$} & {\small{}37384} & {\small{}42532} & {\small{}130} & {\small{}42533/12274} & {\small{}583/546} & {\small{}639} & {\small{}649} & {\small{}8} &{\small{}233/188}& {\small{}9/\textbf{6}}  \tabularnewline
			\hline
			{\small{}$25$} & {\small{}112562} & {\small{}123551} & {\small{}278} & {\small{}174845/21127} & {\small{}1017/1131} & {\small{}4419} & {\small{}4486} & \textbf{\small{}39} &{\small{}5833/1836} &{\small{}93/60}  \tabularnewline
			\hline
			{\small{}$38$}  & {\small{}155503} & {\small{}163197} & {\small{}401} & {\small{}500000*/71991} & {\small{}1208/1032} & {\small{}12636} & {\small{}12101} & {\small{}97} &{\small{}26258*/8957} &{\small{}168/\textbf{95}}  \tabularnewline
			\hline
			{\small{}$50$}  & {\small{}79752} & {\small{}79064} & {\small{}247} & {\small{}172535/12450} & {\small{}730/615} & {\small{}4406} & {\small{}5264} & {\small{}44} &{\small{}5503/1033} & {\small{}65/\textbf{39}} \tabularnewline
			\hline
		\end{tabular}
		\par\end{centering}
	\caption{Numerical results for AG, NM, UP, NC, AD, ACT and AC}\label{tab:t4}
\end{table}
\vspace{-5mm}

\begin{table}[H]
	\begin{centering}
		\begin{tabular}{|>{\centering}p{1.5cm}|>{\centering}p{1cm}>{\centering}p{1cm}>{\centering}p{1cm}|>{\centering}p{1cm}>{\centering}p{1cm}>{\centering}p{1cm}|}
			\hline 
			{$M$}
			& \multicolumn{3}{c|}{AC} & \multicolumn{3}{c|}{ACT}
			\tabularnewline  
			\cline{2-7}
			& {\small{}Max} & {\small{}Avg} & {\small{}Good} & {\small{}Max} & {\small{}Avg} & {\small{}Good}
			\tabularnewline
			\hline 
			{\small{}13} & {\small{}0.25} & {\small{}0.05} & {\small{}67\%} & {\small{}0.06} & {\small{}0.05} & {\small{}71\%}
			\tabularnewline \hline
			{\small{}25} & {\small{}0.47} & {\small{}0.06} & {\small{}65\%} & {\small{}0.08} & {\small{}0.06} & {\small{}69\%}
			\tabularnewline
			\hline
			{\small{}38} & {\small{}0.34} & {\small{}0.07} & {\small{}63\%} & {\small{}0.10} & {\small{}0.07} & {\small{}66\%}
			\tabularnewline \hline
			{\small{}50} & {\small{}0.18} & {\small{}0.07} & {\small{}71\%} & {\small{}0.11} & {\small{}0.07} & {\small{}74\%}
			\tabularnewline
			\hline
		\end{tabular}
		\par\end{centering}
	\caption{AC and ACT statistics}\label{tab:t4-1}
\end{table}
\vspace{-3mm}



In summary, computational results demonstrate that: i) the computed
average curvature of AC is small compared with $ M $ and the computed maximum curvature; 
ii) the percentage of good iterations of AC lies in a suitable
range; and iii) AC is either the best method or close to the best one in terms of running time.

\subsection{Sparse PCA} \label{subsec:SPCA}

This subsection considers a penalized version of the sparse PCA problem,
namely,
\begin{equation}\label{eq:penalty}
\min _{X, Y\in \R^{p\times p}}-\inner{\hat \Sigma}{X}_{F}+\frac{\mu}{2}\|X\|_F^2+Q_{\lam, b}(Y)+\lam \|Y\|_1+\frac{\beta}{2}\|X-Y\|_F^2+I_{\mathcal{F}^{r}}(X),
\end{equation}
where 
the dataset consists of
an empirical covariance matrix $ \hat \Sigma \in \R^{p\times p}$, two regularization parameters $ \mu>0$ and $\lam>0$,
a penalty parameter $\beta>0$ and two scalars $b>0$ and $ r\in \mathbb{N}_+ $. Moreover, $\|\cdot\|_1$ and $ Q_{\lambda, b}(\cdot)$
are the matrix $1$-norm and a decomposable nonconvex penalty function defined as
\[
\|Y\|_1:=\sum_{i, j=1}^{p}\left|Y_{i j}\right|, \quad
Q_{\lambda, b}(X):=\sum_{i, j=1}^{p} q_{\lambda, b}\left(X_{i j}\right)
\]
where
\[
q_{\lam, b}(t) := \left\{ \begin{array}{cc}  
-\frac{t^{2}}{2 b}, & \mbox{if $|t| \leq b \lambda$};  \\ [.1in]
\frac{b \lambda^{2}}{2}-\lambda|t|, & \mbox{otherwise} 
\end{array} \right.
\]
and $ I_{\mathcal{F}^{r}}(\cdot) $ is the indicator function of the Fantope
\[
\mathcal{F}^{r}:=\{X\in {\cal S}^n : 0\preceq X \preceq I \text { and } \operatorname{tr}(X)=r\}.
\]
Clearly, problem \eqref{eq:penalty} is an instance of the nonconvex SCO problem \eqref{eq:PenProb2Intro} where
\[
f(X,Y)=-\inner{\hat \Sigma}{X}_{F}+\frac{\mu}{2}\|X\|_F^2+Q_{\lam,b}(Y)+\frac{\beta}{2}\|X-Y\|_F^2, \quad h(X,Y)=I_{\mathcal{F}^{r}}(X) + \lam \|Y\|_1.
\]
Moreover, it is easy to see that the pair
\begin{equation}\label{eq:M-SPCA}
(M,m)=\left( \max\left\lbrace  \mu+2\beta,\frac 1b \right\rbrace ,\frac 1b\right) 
\end{equation}
satisfies assumption (A2).

We discuss how synthetic datasets are generated. Let $ \Sigma\in \R^{p\times p} $ be an unknown covariance matrix and $ X^* $ be the projection matrix onto the $ r $-dimensional principal subspace of $ \Sigma $.
In the sparse PCA problem, we seek an $ s $-sparse approximation $ X $ of $ X^* $ in the sense that $ \|\text{diag}(X)\|_0\le s $, where $ s\in \mathbb{N}_+$. 
We generate four datasets by designing four covariance matrices $ \Sigma $ as described in \cite{gu2014sparse} and list all required parameters in Table \ref{tab:t5}.
For each covariance matrices $ \Sigma $, we sample $ n = 80  $ i.i.d. observations from the normal distribution $ {\mathcal N} (0, \Sigma) $ and then calculate the sample covariance matrix $ \hat \Sigma $.

\begin{table}[H]
	\begin{centering}
		\begin{tabular}{>{\centering}p{1.2cm}|>{\centering}p{1cm}|>{\centering}p{1cm}|>{\centering}p{1cm}|>{\centering}p{1cm}|>{\centering}p{1cm}|>{\centering}p{1cm}|>{\centering}p{1cm}}
			\hline 
			{dataset} & $ s $ & $ r $ & $ p $ & $ b $ & $ \beta $ & $ \mu $ & $ \lam $ \tabularnewline
			\hline 
			{\small\text{I}}  & {\small{}10} & {\small{}5} & {\small{}1200} & {\small{}3} & {\small{}0.33} & {\small{}1.67} & {\small{}0.25}\tabularnewline \hline
			{\small\text{II}} & {\small{}10} & {\small{}5} & {\small{}1200} &  {\small{}3} & {\small{}0.33} & {\small{}3.33} & {\small{}1}\tabularnewline \hline
			{\small\text{III}}  & {\small{}5} & {\small{}1} & {\small{}1200} & {\small{}3} & {\small{}30} & {\small{}3} & {\small{}5}\tabularnewline \hline
			{\small\text{IV}}  & {\small{}5} & {\small{}1} & {\small{}1200} & {\small{}3} & {\small{}30} & {\small{}0.67} & {\small{}1}\tabularnewline 
			\hline  
		\end{tabular}
		\par\end{centering}
	\caption{Synthetic datasets for the sparse PCA problem}\label{tab:t5}
\end{table}
\vspace{-3mm}

All seven methods are started from the same initial point $ (X_0,Y_0)$ that are chosen as follows.
For datasets I and II, we set $ X_0=Y_0$ to be a diagonal matrix with the first five diagonal entries equal to
1 and the other entries equal zero.
For datasets III and IV, we set $ X_0=Y_0 $ with the first diagonal entry being 1 and any other entries being 0.
We observe that the initial points were chosen differently so as to guarantee that 
they are feasible (i.e., lie in $\dom h$) for their respective instances.
The parameter $ \alpha$ is set to $ 0.5 $ in both AC and ACT.

Numerical results of the seven methods are given in Table \ref{tab:t6} and some statistic measures of AC and ACT are given in Table \ref{tab:t6-1}. The explanation of
their columns excluding the first one is the same as those of Tables \ref{tab:t4} and \ref{tab:t4-1}, respectively.
Their first columns differ from those of Tables \ref{tab:t4} and \ref{tab:t4-1} in that the value of $M$ is computed according to \eqref{eq:M-SPCA}.
The best objective function values obtained by all seven methods are not reported since they are essentially the same on all instances.
The number of resolvent evaluations is 1 in NC, 2 in AG, AC and ACT, 1 or 2 in NM, 1 on average in AD, and 3 on average in UP.
The bold numbers highlight the method that has the best performance in an instance of the problem.

%
%

\begin{table}[H]
	\begin{centering}
		\begin{tabular}{|>{\centering}p{.8cm}|>{\centering}p{0.6cm}>{\centering}p{.6cm}>{\centering}p{.6cm}>{\centering}p{1.2cm}>{\centering}p{1.4cm}|>{\centering}p{1cm}>{\centering}p{1cm}>{\centering}p{1cm}>{\centering}p{2cm}>{\centering}p{2cm}|}
			\hline 
			{$M$}
			& \multicolumn{5}{c|}{Iteration Count} & \multicolumn{5}{c|}{Running Time (s)}
			\tabularnewline  
			\cline{2-11}
			&  {\small{}AG} & {\small{}NM} & {\small{}UP}& {\small{}NC/AD} & {\small{}ACT/AC} & {\small{}AG} & {\small{}NM} & {\small{}UP}& {\small{}NC/AD} & {\small{}ACT/AC}
			\tabularnewline
			\hline 
			{\small{}$2.33$} & {\small{}21} & {\small{}{18}} & {\small{}7} & {\small{}15/31} & {\small{}18/15} & {\small{}8.63} & {\small{}{4.96}} & {\small{}6.71} &{\textbf{\small{}4.50}/10.70}& {\small{}9.70/7.33}  \tabularnewline
			\hline
			{\small{}$4$} & {\small{}7} & {\small{}{9}} & {\small{}8} & {\small{}13/12} & {\small{}9/7} & {\small{}10.08} & \textbf{\small{}{2.73}} & {\small{}7.55} &{\small{}4.42/4.01} &{\small{}4.66/3.94}  \tabularnewline
			\hline
			{\small{}$63$}  & {\small{}32} & {\small{}{43}} & {\small{}18} & {\small{}81/48} & {\small{}43/27} & {\small{}19.91} & {\small{}{12.06}} & {\small{}17.61} &{\small{}22.54/16.05} &{\small{}24.08/\textbf{12.04}}  \tabularnewline
			\hline
			{\small{}$60.67$}  & {\small{}35} & {\small{}{46}} & {\small{}17} & {\small{}84/52} & {\small{}48/31} & {\small{}19.01} & {\small{}{14.28}} & {\small{}16.97} &{\small{}24.31/17.05} & {\small{}26.70/\textbf{12.51}} \tabularnewline
			\hline
		\end{tabular}
		\par\end{centering}
	\caption{Numerical results for AG, NM, UP, NC, AD, ACT and AC}\label{tab:t6}
\end{table}
\vspace{-5mm}

\begin{table}[H]
	\begin{centering}
		\begin{tabular}{|>{\centering}p{1.5cm}|>{\centering}p{1cm}>{\centering}p{1cm}>{\centering}p{1cm}|>{\centering}p{1cm}>{\centering}p{1cm}>{\centering}p{1cm}|}
			\hline 
			{$M$}
			& \multicolumn{3}{c|}{AC} & \multicolumn{3}{c|}{ACT}
			\tabularnewline  
			\cline{2-7}
			& {\small{}Max} & {\small{}Avg} & {\small{}Good} & {\small{}Max} & {\small{}Avg} & {\small{}Good}
			\tabularnewline
			\hline 
			{\small{}2.33} & {\small{}2.00} & {\small{}0.72} & {\small{}67\%} & {\small{}2.83} & {\small{}0.90} & {\small{}83\%}
			\tabularnewline \hline
			{\small{}4} & {\small{}3.67} & {\small{}3.41} & {\small{}71\%} & {\small{}5.02} & {\small{}5.02} & {\small{}89\%}
			\tabularnewline
			\hline
			{\small{}63} & {\small{}44.41} & {\small{}31.12} & {\small{}89\%} & {\small{}43.59} & {\small{}43.55} & {\small{}98\%}
			\tabularnewline \hline
			{\small{}60.67} & {\small{}36.00} & {\small{}28.26} & {\small{}94\%} & {\small{}41.55} & {\small{}41.39} & {\small{}98\%}
			\tabularnewline
			\hline
		\end{tabular}
		\par\end{centering}
	\caption{AC and ACT statistics}\label{tab:t6-1}
\end{table}
\vspace{-3mm}



In summary, computational results demonstrate that: i) the computed
average curvature of AC is close to the computed maximum curvature; 
ii) the percentage of good iterations of AC lies in a suitable
range; and iii) AC is either the best method or close to the best one in terms of running time.

\subsection{Matrix Completion} \label{subsec:MC}

This subsection focuses on a constrained version of the nonconvex low-rank matrix completion problem.
Before stating the problem, we first give a few definitions.
Let $\Omega$ be a subset of $\{1,\ldots,l\} \times \{1,\ldots,n\}$
and let $ \Pi_\Omega $ denote the linear operator that maps a matrix $ A $ to the matrix whose entries
in $\Omega$ have the same values of the corresponding ones in $A$ and
whose entries outside of $\Omega$ are all zero.
Also, for given parameters $\beta>0$ and $\theta>0$,
let $p:\R\to\R_+ $ denote the log-sum penalty defined as
\[
p(t)= p_{\beta,\theta}(t) := \beta\log\left( 1+\frac{|t|}{\theta}\right) .
\]

The constrained version of the nonconvex low-rank matrix completion problem considered in this subsection is
\begin{equation}\label{eq:MC}
\min_{Z\in\R^{l\times n}}\left\lbrace  \frac12 \|\Pi_\Omega(Z-O)\|_F^2+
\mu\sum_{i=1}^{r}p(\sigma_i(Z)) : Z \in {\cal B}_R \right\rbrace
\end{equation}
where $R$ is a positive scalar, ${\cal B}_R := \{ Z \in \R^{l \times n} : \|Z\|_F \le R\}$,
$O \in \R^{\Omega}$ is an incomplete observed matrix, $ \mu>0 $ is a parameter,
$ r := \min \{l,n\}$ and $\sigma_i(Z)$ is the $ i $-th singular value of $ Z $.
The above problem differs from the one considered in \cite{yao2017efficient} in that it adds the constraint
$\|Z\|_F \le R$ to the latter one.

The matrix completion problem in \eqref{eq:MC} is equivalent to
\begin{equation}\label{eq:MC struc}
\min_{Z\in\R^{l\times n}} f(Z)+h(Z),
\end{equation}
where
\begin{align*}
f(Z)=\frac12 \|\Pi_\Omega(Z-O)\|_F^2+ \mu\sum_{i=1}^{r}[p(\sigma_i(Z))-p_0\sigma_i(Z)], \\
h(Z)=\mu p_0\|Z\|_* + I_{{\cal B}_R}(Z), \quad p_0=p'(0)=\frac{\beta}{\theta}
\end{align*}
and $ \|\cdot\|_*$ denotes the nuclear norm defined as
$ \|\cdot\|_* := \sum_{i=1}^r \sigma_i(\cdot)$.
Note that the inclusion of the constraint $ Z \in {\cal B}_R $ in \eqref{eq:MC} implies that the above composite function $h$ has
bounded domain and hence satisfies assumption (A3).
It is proved in \cite{yao2017efficient} that the second term in the definition of $f$, i.e.,
$ \mu\sum_{i=1}^{r}[p(\sigma_i(\cdot))-p_0\sigma_i(\cdot)]$,
is concave and $ 2\mu\tau $-smooth
where $ \tau=\beta/\theta^2 $, so $ f $ is nonconvex and smooth. Since $ h $ is convex and nonsmooth, the problem in \eqref{eq:MC struc} falls into the general class of nonconvex SCO problems \eqref{eq:PenProb2Intro}. It is easy to see that the pair
\begin{equation}\label{eq:M-MC}
(M,m)=(\max\{1,2\mu\tau\},2\mu\tau )
\end{equation}
satisfies assumption (A2).

We use the {\it MovieLens} dataset\footnote{http://grouplens.org/datasets/movielens/}
to obtain the observed index set $\Omega$ and the incomplete observed matrix $ O $.
The dataset includes a sparse matrix with 100,000 ratings of \{1,2,3,4,5\} from 943 users on 1682 movies, namely $ l=943 $ and $ n=1682 $.
The radius $ R $ is chosen as the Frobenius norm of the matrix of size $ 943 \times 1682 $
containing the same entries as $ O $ in $ \Omega $ and 5 in the entries outside of $ \Omega $.

We start all seven methods from the same initial point $ Z_0 $ that is sampled from the standard Gaussian distribution and is within ${\cal B}_R$.
The parameter $ \alpha$ is set to $ 0.5 $ in AC and $0.1$ in ACT.

Numerical results of the seven methods are given in Table \ref{tab:t7} and some statistic measures of AC and ACT are given in Table \ref{tab:t7-1}. 
The format of Table \ref{tab:t7} is similar to that of Table \ref{tab:t6} with the exception that the second to sixth columns provide the function values of \eqref{eq:MC} at the last iteration and the numbers of iterations for all seven methods.
Note that the first columns of Tables \ref{tab:t7} and \ref{tab:t7-1} give the value of $ M $ computed according to \eqref{eq:M-MC}.
The number of resolvent evaluations is 1 in NC, 2 in AG, AC and ACT, 1 or 2 in NM, 1 on average in AD, and 3 on average in UP.
The bold numbers highlight the method that has the best performance in an instance of the problem.




\begin{table}[H]
	\begin{centering}
		\begin{tabular}{|>{\centering}p{.6cm}|>{\centering}p{.8cm}>{\centering}p{.8cm}>{\centering}p{.8cm}>{\centering}p{1.7cm}>{\centering}p{1.7cm}|>{\centering}p{.6cm}>{\centering}p{.6cm}>{\centering}p{.6cm}>{\centering}p{1.7cm}>{\centering}p{1.8cm}|}
			\hline 
			{\small{}$M$} & \multicolumn{5}{c|}{\makecell{\small{}Function Value / \\ \small{}Iteration Count}}
			& \multicolumn{5}{c|}{\small{}Running Time (s)}
			\tabularnewline  
			\cline{2-11}
			& {\small{}AG}& {\small{}NM} & {\small{}UP} & {\small{}NC/AD} & {\small{}ACT/AC}  & {\small{}AG} & {\small{}NM} & {\small{}UP} & {\small{}NC/AD} & {\small{}ACT/AC}
			\tabularnewline
			\hline 
			{\small{}4.4} & {\small{}2257\\3856} & {\small{}\textbf{1809}\\1036} & {\small{}2605\\521} & {\small{}2628/2625\\  {4780/1674}} & {\small{}2252/2288\\  {5420/765}}
			& {\small{}4568} & {\small{}1033} & {\small{}1545} & {\small{}3925/1946} & {\small{}5803/\textbf{833}}
			\tabularnewline \hline
			{\small{}8.9} & {\small{}3886\\ 9158} & {\small{}\textbf{3359}\\1617} & {\small{}4261\\576} & {\small{}4246/4203\\ {9751/1794}} & {\small{}3846/3884\\  {8726/968}}
			& {\small{}10251} & {\small{}1605} & {\small{}1621} & {\small{}7901/1930} & {\small{}8806/\textbf{1065}}
			\tabularnewline
			\hline
			{\small{}20} & {\small{}4282\\22902} & {\small{}\textbf{3635}\\2875} & {\small{}4637\\676} & {\small{}4641/4582\\ {22259/2209}} & {\small{}4282/4267\\  {13031/1079}}
			& {\small{}29274} & {\small{}2836} & {\small{}1914} & {\small{}15912/2364} & {\small{}13869/\textbf{1200}}
			\tabularnewline \hline
			{\small{}30} & {\small{}5967\\37032} & {\small{}\textbf{5237}\\3717} & {\small{}6753\\606} & {\small{}6380/6293\\ {32223/1963}} & {\small{}5963/5975\\  {18267/1085}}
			& {\small{}41673} & {\small{}4182} & {\small{}1628} & {\small{}22265/2104} & {\small{}19913/\textbf{1214}}
			\tabularnewline
			\hline
		\end{tabular}
		\par\end{centering}
	\caption{Numerical results for AG, NM, UP, NC, AD, ACT and AC}\label{tab:t7}
\end{table}
\vspace{-5mm}

\begin{table}[H]
	\begin{centering}
		\begin{tabular}{|>{\centering}p{1.5cm}|>{\centering}p{1cm}>{\centering}p{1cm}>{\centering}p{1cm}|>{\centering}p{1cm}>{\centering}p{1cm}>{\centering}p{1cm}|}
			\hline 
			{$M$}
			& \multicolumn{3}{c|}{AC} & \multicolumn{3}{c|}{ACT}
			\tabularnewline  
			\cline{2-7}
			& {\small{}Max} & {\small{}Avg} & {\small{}Good} & {\small{}Max} & {\small{}Avg} & {\small{}Good}
			\tabularnewline
			\hline 
			{\small{}4.4} & {\small{}1.00} & {\small{}0.31} & {\small{}96\%} & {\small{}1.00} & {\small{}0.45} & {\small{}99\%}
			\tabularnewline \hline
			{\small{}8.9} & {\small{}1.00} & {\small{}0.28} & {\small{}94\%} & {\small{}1.39} & {\small{}0.48} & {\small{}99\%}
			\tabularnewline
			\hline
			{\small{}20} & {\small{}0.99} & {\small{}0.25} & {\small{}91\%} & {\small{}2.65} & {\small{}0.72} & {\small{}99\%}
			\tabularnewline \hline
			{\small{}30} & {\small{}0.97} & {\small{}0.23} & {\small{}89\%} & {\small{}4.36} & {\small{}1.13} & {\small{}96\%}
			\tabularnewline
			\hline
		\end{tabular}
		\par\end{centering}
	\caption{AC and ACT statistics}\label{tab:t7-1}
\end{table}
\vspace{-3mm}



In summary, computational results demonstrate that: i) the computed
average curvature for AC is small compared with $ M $ and the computed maximum curvature;
ii)
the percentage of good iterations of AC lies in a suitable
range; and iii) AC has the best performance in terms of running time.
Although AC uses the least amount of time to terminate, NM finds solutions with the smallest objective function values.

\subsection{Nonnegative Matrix Factorization} \label{subsec:NMF}


This subsection focuses on the following NMF problem
\begin{equation}\label{eq:NMF}
\min \left \{f(X,Y):=\frac{1}{2}\|A-XY\|_F^2:X\ge0,Y\ge0\right\},
\end{equation}
where $ A\in \R^{n\times l} $, $ X\in \R^{n\times p} $ and $ Y\in \R^{p\times l} $, which have been thoroughly
studied in the literature (see e.g.,\ \cite{gillis2014and,lee1999learning}).

This subsection reports the efficiency of directly using all seven methods to
solve \eqref{eq:NMF} without making use of its two-block structure. We use the facial image dataset provided by AT\&T Laboratories Cambridge\footnote{https://www.cl.cam.ac.uk/research/dtg/attarchive/facedatabase.html} to construct
the matrix $A$. More specifically, this dataset consists of 400 images, and each of those contains $ 92\times112 $ pixels with 256 gray levels per pixel. It results in an $ n \times l = 10, 304\times400 $ matrix $ A $ whose columns are the vectorized images. The dimension $ p $ is set to 20.

We start all seven methods from the same initial point $ (X_0,Y_0)= (\mathbf{1}^{n\times p}/(np), \mathbf{1}^{p\times l}/(pl)) $, where $ \mathbf{1}^{n\times p} $ and $ \mathbf{1}^{p\times l} $ are matrices of all ones of sizes
$ n\times p $ and $ p\times l $, respectively. We estimate
$ M $ in \eqref{ineq:curv} as $ M=100\times {\cal C}\left( (X_0,Y_0), (0,0)\right) $
where $ {\cal C}(\cdot,\cdot) $ is defined in \eqref{eq:cal C}.
Since no sharper $m<M$ satisfying the first inequality in \eqref{ineq:curv}
is known, we simply set $m=M$.
The parameter $ \alpha$ is set to $ 0.7 $ in both AC and ACT.

Numerical results for the seven methods are given in Table \ref{tab:t8}. The bold numbers highlight the method that has the best performance in the problem. 
The best objective function values obtained by all seven methods are not reported since they are essentially the same.




\begin{table}[H]
	\begin{centering}
		\begin{tabular}{|>{\centering}p{2cm}|>{\centering}p{3cm}|>{\centering}p{3cm}|}
			\hline 
			{Method} & {Iteration Count}& {Running time(s)}\tabularnewline
			\hline 
			{\small\text{AG}}  & {\small{}786} & {\small{}73.03} \tabularnewline
			\hline
			{\small\text{NM}} & {\small{}162} & {\small{}14.91} \tabularnewline
			\hline
			{\small\text{UP}} & {\small{}37} &  {\small{}11.12} \tabularnewline
			\hline
			{\small\text{NC}}  & {\small{}656} &  {\small{}41.67} \tabularnewline
			\hline
			{\small\text{AD}} & {\small{}44} &  {\small{}5.21} \tabularnewline
			\hline
			{\small\text{ACT}}  & {\small{}41} & {\small{}6.54} \tabularnewline
			\hline  
			{\small\text{AC}}  & {\small{}36} & \textbf{\small{}4.70} \tabularnewline
			\hline  
		\end{tabular}
		\par\end{centering}
	\caption{Numerical results for AG, NM, UP, NC, AD, ACT and AC}\label{tab:t8}
\end{table}
\vspace{-5mm}

\section{Concluding remarks}\label{sec:conclusion}
This paper presents an average curvature accelerated composite gradient method, namely, the AC-ACG method, for solving the N-SCO problem which is based on the average of all observed curvatures.
More specifically, as opposed to other ACG variants,
which use a known Lipschitz constant or a backtracking procedure that searches for a good upper curvature $ M_k $, AC-ACG
uses the average of all observed curvatures to compute $M_k$
(see \eqref{eq:M}) and always accepts the first
computed iterate according to \eqref{eq:update}
no matter whether $ M_k $ is good or not.
A nice feature of AC-ACG is that its convergence rate bound
is expressed in terms of $ M_k $ rather than an upper curvature $ M\ge \bar M $.



We now discuss some possible extensions of this paper. 
First, numerical results show that the AC variant,
which computes $ C_k $ 
as in \eqref{eq:AC}, performs substantially better than previous ACG variants as well as the ACT variant,
which is closer to the main method analyzed in
this paper, namely, AC-ACG.
However, convergence rate analysis of AC (possibly with $\gamma$ and $\alpha$ satisfying \eqref{eq:alpha}) is an
interesting open problem.
Second, the AC-ACG method performs two resolvent evaluations of $ h $ per iteration. It would be desirable to develop
 AC-ACG variants
which only perform one resolvent evaluation of $ h $ per iteration.
Third, the analysis of AC-ACG
assumes that assumption (A3) holds,
i.e., $\dom h$  is bounded.
It would be interesting to
develop a variant of it with
a provably iteration-complexity
similar to the one in this paper
without assuming (A3).


\section{Acknowledgements}
We are grateful to Guanghui Lan and Saeed Ghadimi for providing the code for the UPFAG method of their paper \cite{LanUniformly}.
We are also grateful to the two anonymous referees and the associate editor Defeng Sun for
providing helpful comments on earlier versions of this manuscript.

\bibliographystyle{plain}
\bibliography{Proxacc_ref}

\appendix
\section{A technical result}\label{Appendix A}

Recall the definition of a good upper curvature of $f$ given above \eqref{ineq:descent}.

\begin{lemma}\label{lem:descent}
	If $ M_k $ is a good upper curvature of $ f $ at $ \tx_k $ and $y_{k+1}=y(\tx_k;M_k)$ where $ y(\cdot;\cdot) $ is defined in \eqref{eq:update}, then
	\begin{equation}\label{ineq:implication}
	\phi(y_{k+1})  \le \phi(\tx_k) - \frac{M_k}{2}\|y_{k+1}-\tx_k\|^2.
	\end{equation}
\end{lemma}
\begin{proof}
	Using the fact that $ M_k $ is a good upper curvature of $ f $ at $ \tx_k $ and \eqref{ineq:descent}, we have
	\begin{equation}\label{ineq:part1}
	\phi(y_{k+1}) \le \ell_f(y_{k+1};\tx_k) + h(y_{k+1}) + \frac{M_k}{2}\|y_{k+1}-\tx_k\|^2.
	\end{equation}
	It follows from the definition of $ y_{k+1} $, \eqref{eq:update} and the fact that the objective function in \eqref{eq:update} is $ M_k $-strongly convex that for every $ u\in \dom h $,
	\[
	\ell_f(u;\tx_k) + h(u) + \frac{M_k}{2}\|u-\tx_k\|^2\ge \ell_f(y_{k+1};\tx_k) + h(y_{k+1}) + \frac{M_k}{2}\|y_{k+1}-\tx_k\|^2 + \frac{M_k}{2}\|u-y_{k+1}\|^2,
	\]
	which together with $ u=\tx_k $ implies that
	\[
	\phi(\tx_k)\ge \ell_f(y_{k+1};\tx_k) + h(y_{k+1}) + M_k\|y_{k+1}-\tx_k\|^2.
	\]
	Now inequality \eqref{ineq:implication} immediately follows from \eqref{ineq:part1} and the above inequality.
\end{proof}

\end{document}